\documentclass[12pt]{article}
\usepackage{geometry}
\usepackage{amsfonts,amssymb} 
\usepackage{theorem}
\usepackage{graphicx}
\usepackage{amssymb}
\usepackage{epstopdf}
\usepackage{xy}
\usepackage{color}
\newtheorem{thm}{Theorem}[section]
\newtheorem{prop}[thm]{Proposition}
\newtheorem{lemma}[thm]{Lemma}
\newtheorem{cor}[thm]{Corollary}

\makeatletter
\def\operatorname#1{\mathop{\operator@font #1}\nolimits}%
\makeatother

\newenvironment{proof}[1][{}]{{ \textsc{Proof{#1}:~}}}{{\hspace*{\fill}$\square$\\}}

\begin{document}

\setcounter{page}{1}
\pagestyle{plain}

\title{Intrinsic reflections in Coxeter systems}
\author{Bernhard M\"uhlherr${}^1$ \and Koji Nuida${}^{23}$}
\date{\normalsize ${}^1$ Mathematisches Institut, JLU Giessen \\
${}^2$ Information Technology Research Institute, National Institute of Advanced Industrial Science and Technology (AIST) \\
${}^3$ Japan Science and Technology Agency (JST) PRESTO Researcher}

\maketitle

\begin{abstract}
Let $(W,S)$ be a Coxeter system and let $s \in S$. We call $s$ a right-angled generator
of $(W,S)$ if $st = ts$ or $st$ has infinite order for each $t \in S$. We call
$s$ an intrinsic reflection of $W$ if $s \in R^W$ for all Coxeter generating sets $R$ of $W$.
We give necessary and sufficient conditions for a right-angled generator $s \in S$ of $(W,S)$
to be an intrinsic reflection of $W$.
\end{abstract}

\section{Introduction}

Let $(W,S)$ be a Coxeter system. We call an element $x \in W$
an {\sl intrinsic reflection of $W$} if $x \in R^W := \{ w^{-1}rw \mid r \in R, w \in W \}$
for all Coxeter generating sets $R$ of $W$. This paper is a contribution
to the following problem.

\smallskip
\noindent
{\bf Problem:} Let $(W,S)$ be a Coxeter system and $s \in S$. Give
necessary and sufficient condition for $s$ to be an intrinsic reflection
of $W$ in terms of the diagram of $(W,S)$.

\smallskip
\noindent
This problem arises naturally in the context of the isomorphism problem
for Coxeter groups which is still open at present. Substantial progress
has been made by Caprace and Przytycki in \cite{CP} in this area.
Combined with their result, the complete solution to the problem above
would provide a characterization of all strongly rigid Coxeter systems
in terms of their diagrams. This will be discussed in more
detail in the paragraph on strong rigidity below.

\bigskip
\noindent
In \cite{FHM} the {\sl finite continuation} of a finite
order element in a finitely generated Coxeter group was introduced.
The notion of an intrinsic reflection was not defined in that paper. However,
its main purpose was to give a criterion which ensures that a generator
$s \in S$ of a Coxeter system $(W,S)$ is an intrinsic reflection of $W$.
Indeed, Part b) of the Main result of \cite{FHM} asserts that $s \in S$ is
an intrinsic reflection of $W$ if its finite continuation coincides
with the subgroup of order 2 generated by $s$.
In a forthcoming paper we intend to give a complete solution of the problem
above and it is indeed the case that the criterion given in \cite{FHM} plays
a central role in our arguments. However, it turns out that
the information provided by this criterion is rather limited if the generator
$s \in S$ of the Coxeter system $(W,S)$ is {\sl right-angled} by which we mean that
the order of $st$ is in $\{ 1,2,\infty \}$ for all $t \in S$. The purpose of
this paper is to solve the above problem for right-angled generators.

There are several reasons
for treating this special case in a separate paper: As already pointed out above,
the general results about the finite continuation do not provide any particularly
deep insights for right-angled generators. In fact, the corresponding information can be
deduced more efficiently by direct ad-hoc arguments. Thus, we do not make use of the
finite continuation here. This has the advantage that we do not have to assume that our
Coxeter systems have finite ranks. Another reason for treating right-angled
generators separately is the fact that there are specific tools which are only needed
in this special case. Indeed, the case of right-angled generators is the only one where
the {\sl blowing down} procedure for Coxeter generating sets comes into play. The latter
has been introduced in \cite{MR} and some of our intermediate results should be compared
with those in \cite{MR}. However, our treatment is completely independent because
our assumptions and goals are quite different from those in \cite{MR}.

\subsection*{Intrinsic reflections and blowing down in spherical Coxeter systems}

Before stating the main result of this paper it is convenient to provide
first some basic information
about intrinsic reflections in finite Coxeter groups.

Let $(W,S)$ be a Coxeter system. We call $(W,S)$ {\sl spherical} if $W$ is a finite group.
A subset $J$ of $S$ is called a direct factor
if $J \neq \emptyset$ and $[J, S \setminus J] = 1$; moreover $(W,S)$ is called irreducible if
$S$ is the only direct factor of $(W,S)$.

The irreducible spherical Coxeter systems are known. We denote
their types (i.e. their diagrams) as
in \cite{Humphreys} with the
only exception that we use here $C_n$ instead of $B_n$.
For each irreducible
spherical Coxeter system $(W,S)$
a description of $W$ as abstract group
is found in Appendix 5 of \cite{Ro} and the following
is a straightforward (but somewhat lengthy) exercise in finite group theory.

\smallskip
\noindent
{\bf Fact 1:} Let $(W,S), (W',S')$ be two irreducible spherical Coxeter systems
such that $W$ is isomorphic to $W'$. Then there exists an isomorphism from $W$ onto
$W'$ mapping $S$ onto $S'$.

\smallskip
\noindent
For each irreducible spherical Coxeter system $(W,S)$ the description of the
automorphism group of the abstract group $W$ is given in Theorem 31 of \cite{FH}.
Combining this with the previous fact one deduces the following.

\smallskip
\noindent
{\bf Fact 2:} Let $(W,S)$ be an irreducible spherical Coxeter system.
\begin{itemize}
\item[(i)]
If the center of $W$ is trivial and $(W,S)$ is not of type $A_5$,
then each $s \in S$ is an intrinsic reflection of $W$.
\item[(ii)]
If $(W,S)$ is of type $A_1$, $H_3$, $E_7$ or $I_2(4k)$ for some $k \in {\bf N}$,
then each $s \in S$ is an intrinsic reflection of $W$.
\item[(iii)] If $(W,S)$ is not covered by (i) or (ii), then
no $s \in S$ is an intrinsic reflection of $W$.
\end{itemize}

\bigskip
\noindent
{\bf Remark:}
Let $3 \leq n \in {\bf N}$ be odd and let $(W,S)$ be a Coxeter system
of type $C_n$ or $I_2(2n)$. Then $W$ can be written as a direct product
$W = A \times W'$ where $A$ is a group of order 2 and $W'$ has a Coxeter generating
set $S'$ such that $(W',S')$ is of type $D_n$ (resp. $I_2(n)$). Thus there is
a Coxeter generating set $R$ of $W$ such that $|R| = |S|+1$.

\bigskip
\noindent
Let $(W,S)$ be a Coxeter system and let $R$ be a Coxeter generating set of $W$.
As the previous remark shows it may happen that $|S| \neq |R|$. The question,
to which extent the abstract group $W$ determines the cardinality of a Coxeter generating
set $R$ was adressed by Mihalik and Ratcliffe in \cite{MR}. They define two procedures
of manipulating Coxeter generating sets which they call {\sl blowing up} and
{\sl blowing down}. These procedures rely on the examples described in
the previous remark.

\bigskip
\noindent
The investigation of intrinsic right-angled generators leads
naturally to the consideration of a situation which is almost equivalent to their
blowing down procedure.
In our context it is convenient to introduce the notion
of
a {\sl blowing down generator} for a right-angled generator $s$ of a Coxeter system
$(W,S)$.
The existence of such a blowing down generator for $s$
ensures that one can find a Coxeter generating set $R$ of $W$ such that
$s$ is not in $R^W$ and hence that $s$ is not an intrinsic reflection of $W$.

\subsection*{The main result}

The precise statement of our main result needs some preparation.
In particular, the definition of a {\sl blowing down generator}
for a right-angled generator in a Coxeter system is somewhat technical.

Let $(W,S)$ be a Coxeter system. For $s,t \in S$ we denote the
order of $st$ by $m_{st}$ and we call $s \in S$ a {\sl right-angled}
generator of $(W,S)$ if $\{ m_{st} \mid t \in S \} \subseteq \{ 1,2,\infty \}$.
Moreover, for any generator $s \in S$ we put $s^{\perp} := \{ t \in S \mid m_{st} = 2 \}$
and $s^{\infty} := \{ t \in S \mid m_{st} = \infty \}$.

Let $s \in S$ be a right-angled generator of $(W,S)$. An {\sl $s$-component} is
an irreducible spherical component of the Coxeter system $(\langle s^{\perp} \rangle, s^{\perp})$.
We call $a \in s^{\perp}$ a {\sl blowing down generator for $s$} if the following conditions
are satisfied:

\begin{itemize}
\item[(BDG1)] If $C$ denotes the irreducible component of $(\langle s^{\perp} \rangle, s^{\perp})$
containing $a$, then $(\langle C \rangle,C)$ is of type $I_2(2k+1)$ or $D_{2k+1}$
for some $1 \leq k \in {\bf N}$; moreover, if $\rho$ denotes the longest element in the
Coxeter system $(\langle C \rangle,C)$, then $b:= \rho a \rho \neq a$.
(Note that
$b \in C$ because $\rho$ normalizes $C$.)
\item[(BDG2)] If $u_0,u_1,\ldots,u_n \in s^{\infty}$ are such that
$u_{i-1}u_i$ has finite order for all $1 \leq i \leq n$, then there exists
an element $x \in \{ a,b \}$ such that $\{ u_i \mid 0 \leq i \leq n \} \subseteq x^{\infty}$.
\end{itemize}

\smallskip
\noindent
{\bf Main result:} Let $(W,S)$ be a Coxeter system of arbitrary rank and
let $s \in S$ be a right-angled generator. Then $s$ is an intrinsic reflection
of $W$ if and only if each $s$-component has trivial center and there is
no blowing down generator for $s$. Moreover, if $s$ is an intrinsic reflection
and if $R$ is a Coxeter generating set of $W$ containing $s$, then
$s$ is a right-angled generator of $(W,R)$.

\bigskip
\noindent
The proof of the main result will be completed in the final section
of this paper.

\subsection*{Some consequences and remarks}

\noindent
{\bf 2-spherical Coxeter systems:}
A Coxeter system $(W,S)$ is called
{\sl 2-spherical} if the order of $st$ is finite for all $s,t \in S$.
Let $(W,S)$ be a 2-spherical, irreducible Coxeter system such that
$W$ is an infinite group. Then
Theorem 1 in \cite{FHM} asserts that any fundamental generator $s \in S$
is an intrinsic reflection of $W$.
Thus, if $(W,S)$ is irreducible and $|W|=\infty$
and if there exists a generator $s \in S$ which is not an intrinsic
reflection of $W$, then there have to exist $t,u \in S$ such that $tu$
has infinite order. In some sense, right-angled generators violate
the 2-sphericity condition for a Coxeter system to large extent.
Thus, heuristically speaking, they are good candidates for generators
which are not
intrinsic reflections $W$.

\bigskip
\noindent
{\bf Right-angled Coxeter-systems:} A Coxeter system $(W,S)$ is called
{\sl right-angled} if each $s \in S$ is right-angled. It is a consequence of
our main result that a generator $s \in S$ of a right-angled Coxeter system
$(W,S)$ is an intrinsic reflection of $W$ if and only if the
Coxeter system $(\langle s^{\perp} \rangle, s^{\perp})$ does not have
a direct factor of type $A_1$. We do not know a reference for
this statement in the literature but it is implicitly in
\cite{Ti88} (for finite $|S|$) and \cite{AC06}.

\bigskip
\noindent
{\bf Blowing-down generators:} We already mentioned that {\sl blowing down}
for Coxeter generating sets was introduced in \cite{MR}. The conditions
for the existence of a blowing down are given in Theorem 3.7 in \cite{MR}.
These conditions are slightly more general than ours.

\subsection*{Strongly rigid Coxeter systems}

We already mentioned that the problem of finding a characterization
of intrinsic fundamental generators in a Coxeter system in terms of
its diagram arises naturally in the context of the isomorphism problem
for Coxeter groups. 
For more information about this open question we refer to \cite{Mu05}
and to \cite{Nuidasurvey} for a more recent account including the infinite rank case.

In the context of the isomorphism problem the notion
of
strong rigidity for Coxeter systems has been introduced in \cite{BMMN}.
A Coxeter system $(W,S)$ is called {\sl strongly rigid} if each Coxeter generating
set $R$ of $W$ is conjugate to $S$ in $W$. If $(W,S)$ is strongly rigid,
then the abstract group $W$ determines the set of its reflections and hence each $s \in S$ is an intrinsic reflection;
moreover, the abstract group $W$ also determines the diagram and
all automorphisms of the abstract group $W$ are inner-by-graph. Thus, strongly
rigid Coxeter systems form an interesting class of Coxeter systems which one
would like to characterize in terms of their diagrams. An important step towards
such a characterization is provided  by a substantial result of Caprace and Przytycki in \cite{CP}.
Their result gives in particular a characterization of all {\sl strongly reflection rigid
Coxeter systems} (see Definition 3.2 in \cite{BMMN}) in terms of their diagrams.
It is a basic fact that a Coxeter system $(W,S)$
is strongly rigid if and only if it is strongly reflection rigid and if each
$s \in S$ is an intrinsic reflection of $W$. Thus, combined with the result of \cite{CP}
the solution of the problem above would provide a characterization of all
strongly rigid Coxeter systems in terms of their diagrams.

\subsection*{Organisation of the paper}

In order to obtain our main result, we have to prove both implications.
One of them is considerably harder to establish and the following
Proposition is an equivalent formulation of it:

\smallskip
\noindent
{\bf Proposition:} Let $(W,S)$ be a Coxeter system of arbitrary rank and
let $s \in S$ be a right-angled generator such that each $s$-component
has trivial center. If there exists a Coxeter generating set $R$ of $W$
such that $s$ is not a reflection of $(W,R)$,
then there exists a blowing down generator for $s$.

\smallskip
\noindent
The proof of the Proposition above will be accomplished
in the final section of the paper.

\bigskip
\noindent
The paper consists of two parts. In the first
part we reduce the proof of the main result to this Proposition. This will
be accomplished at the end of Section \ref{rightangledsec}.
Then, in the considerably more technical second part, we first provide several additional
tools that will be only needed in the proof
of the Proposition.

\smallskip
\noindent
The main tools that we shall use are the following:

\smallskip
\noindent
{\bf A detailed analysis of spherical Coxeter systems:} We shall need some specific
information about Krull-Remak-Schmidt decompositions of finite Coxeter groups.
Furthermore, for the Coxeter groups of type $C_n,I_2(n)$ where $3 \leq n \in {\bf N}$
is odd, we have to study all Coxeter generating sets. Later on, the outcome
of this analysis will lead us to three different cases in the proof of the
Proposition above. These will be treated in Sections \ref{sectioncasebarD}, \ref{sectioncaseD}
and \ref{sectioncaseI} separately.

\smallskip
\noindent
{\bf The Coxeter complex:} We shall consider the Cayley graph of a Coxeter system
$(W,S)$ in order to prove Proposition \ref{keyproposition} which is the key step in the proof
of the Proposition above. The Cayley graph is a building and we shall use
techniques from the theory of buildings. It is for this reason, that we
call the Cayley graph the {\sl Coxeter complex} of $(W,S)$.

\smallskip
\noindent
{\bf Conjugacy theorems and Property FA:} Let $(W,S)$ be a Coxeter system
and $J \subseteq S$. We call $J$ spherical (resp. 2-spherical) if
$\langle J \rangle$ is a finite group (resp. if $st$ has finite order for all $s,t \in J$).
In the proof of Proposition \ref{keyproposition} we shall use the result that each
finite subgroup of $W$ is conjugate to a subgroup of $\langle J \rangle$
for some spherical $J \subseteq S$. There is also a characterization of
the subgroups $U$ of $W$ which are conjugate to subgroups $\langle J \rangle$
for 2-spherical $J \subseteq S$. It turns out that this is the case if and only if $U$
is an FA-group. This will be explained and applied in Section \ref{sectioncaseI}.

\bigskip
\noindent
{\bf Acknowledgements:} Much of the research was undertaken while the
first author was invited at the University of Tokyo. Both authors would like to thank this
institution, and in particular Professor Itaru Terada,
for providing us the possibility
to work together at the Tambara Institute for Mathematical Sciences.
Furthermore, we are grateful to both referees for their valuable comments.

\section{Coxeter systems}

In this paper we allow Coxeter systems of arbitrary rank. Whenever
'finite rank' is required it will be mentioned explicitly. Dropping the
assumption of finite rank does not lead to any serious problems for
most of the basic results on Coxeter groups that we shall use here.
But there are also exceptions to this since the corresponding statements have to be
modified in order to hold also in the infinite rank case (e.g. parabolic closure
of a subgroup) and their proofs need some additional argument. Moreover, even if
a result is also valid in the infinite rank case, it is not obvious to find
it in this generality in the literature, since finite rank is often explicitly required
or tacitly assumed. In this section we collect the basic results on Coxeter systems
that will be used in later and sketch their proofs in the infinite rank case if
finite rank is not assumed.

\medskip
\noindent

Let $(W,S)$ be a Coxeter system. For $s,t \in S$ we denote the order
of $st$ by $m_{st}$. The {\sl length} of an element $w \in W$
with respect to the generating set $S$ is denoted by $\ell(w)$.
A subset $J$ of $S$ is called a {\sl direct factor of $(W,S)$}
if $J \neq \emptyset$ and $[J,S \setminus J] = 1$. The Coxeter system $(W,S)$ is
called {\sl irreducible} if $S$ is the only direct factor of $(W,S)$.
We put $S^W := \{ w^{-1}sw \mid s \in S,w \in W \}$.
The elements of $S^W$ are called the {\sl reflections of $(W,S)$}.

\smallskip
\noindent
In the following lemma we recall several basic facts about
Coxeter systems. We shall often use them in the sequel without
explicitly referring to this lemma.

\begin{lemma} \label{Coxeterbasics}
Let $(W,S)$ be a Coxeter system and $J \subseteq S$. Then
$(\langle J \rangle,J)$ is a Coxeter system and
$S^W \cap \langle J \rangle = J^{\langle J \rangle}$. Moreover,
if $\ell_J:\langle J \rangle \rightarrow {\bf N}$
denotes its length function, then $\ell_J = \ell \mid_{\langle J \rangle}$.
Finally, if $K \subseteq S$, then $\langle J \rangle \cap \langle K \rangle = \langle J \cap K \rangle$.
\end{lemma}

\begin{proof} See, for instance, Theorem 5.5 in \cite{Humphreys} for the case where
$(W,S)$ has finite rank. It is straightforward to reduce the infinite rank case to
the finite rank case.
\end{proof}

Let $(W,S)$ be a Coxeter system and let $J \subseteq S$. An {\sl irreducible component of $J$}
is a direct factor $C$ of $(\langle J \rangle, J)$ such that $(\langle C \rangle,C)$
is irreducible.

\begin{lemma} \label{dirdecbasic}
Let $(W,S)$ be a Coxeter system and let $J \subseteq S$ be a direct factor of $(W,S)$.

\begin{itemize}
\item[(i)] If $K := S \setminus J$ then $W = \langle J \rangle \times \langle K \rangle$
and $\ell(uv) = \ell(u) + \ell(v)$ for all $u \in \langle J \rangle$ and $v \in \langle K \rangle$.
\item[(ii)] If $K \subseteq J$ is a direct factor of $(\langle J \rangle, J)$, then
$K$ is a direct factor of $(W,S)$ as well. In particular, the irreducible
components of $(W,S)$ form a partition of $S$.
\item[(iii)] If there are finitely many irreducible components $C_1,\ldots,C_m$
of $(W,S)$,
then $W = \oplus_{1 \leq i \leq m} \langle C_i \rangle$; moreover, if $v_i \in \langle C_i \rangle$
for $1 \leq i \leq m$, then $\ell(v_1 v_2 \ldots v_m) = \sum_{1 \leq i \leq m} \ell(v_i)$.
\end{itemize}
\end{lemma}

\begin{proof}
Assertion (i) follows from the previous Lemma and Assertion (ii) is straightforward from the
definitions and Assertion (iii) follows from Assertions (i) and (ii) by induction on $m$.
\end{proof}

\subsection*{Spherical Coxeter systems}

A Coxeter system $(W,S)$ is called {\sl spherical} if $W$ is a finite group.
A subset $J$ of $S$ is called {\it spherical} if $(\langle J \rangle,J)$
is a spherical Coxeter system. Note that any spherical Coxeter system
is always of finite rank.

\begin{lemma} \label{longestelt}
Let $(W,S)$ be a spherical Coxeter system with $S \neq \emptyset$. Then there
exists a unique element $\rho \in W$ such that
$\ell(\rho) \geq \ell(w)$ for all $w \in W$. The element $\rho$ is
an involution and $S^{\rho} = S$.
\end{lemma}

\begin{proof}
See, for instance, Exercise 2 of Paragraph 5.6 in \cite{Humphreys}.
\end{proof}

Let $(W,S)$ be a spherical Coxeter system. Then we call the unique
element $\rho \in W$ of Lemma \ref{longestelt} {\sl the longest element} of $(W,S)$.

\smallskip
\noindent
Let $(W,S)$ be a Coxeter system. Then
$(W,S)$ is called of {\sl $(-1)$-type} if it is spherical
and if its longest element is contained in the center of $W$.
A subset $J$ of $S$ is said to be of {\sl $(-1)$-type} if
the Coxeter system $(\langle J \rangle,J)$
is of $(-1)$-type.

\begin{lemma} \label{finitecenter}
Let $(W,S)$ be an irreducible spherical Coxeter system and let $\rho \in W$ be
the longest element in $(W,S)$. If $|Z(W)| > 1$, then
$Z(W) = \langle \rho \rangle$. In particular,
$(W,S)$ is of $(-1)$-type if and only if $Z(W) = \langle \rho \rangle$.
\end{lemma}

\begin{proof} This is deduced from the geometric representation of $(W,S)$. See
for instance Exercise 1 of Paragraph 6.3 in \cite{Humphreys}.
\end{proof}

\begin{lemma} \label{longesteltoffinite}
Let $(W,S)$ be a Coxeter system of finite rank. Then $(W,S)$
is spherical (resp. of $(-1)$-type) if and only if all irreducible
components of $(W,S)$ are spherical (resp. of $(-1)$-type).

Suppose that $(W,S)$ is spherical and let $\rho$ be the longest
element of $(W,S)$. Let $C_1, \ldots ,C_m$ be the irreducible
components of $(W,S)$ and let $\rho_i$ be the longest element
of $(\langle C_i \rangle, C_i)$ for $1 \leq i \leq m$.
Then $\rho = \rho_1 \rho_2 \ldots \rho_m$ and
$\ell(\rho) = \sum_{i = 1}^m \ell(\rho_i)$.
If $(W,S)$ is of $(-1)$-type, then the center of $W$
is an elementary abelian subgroup of $W$ of order $2^m$
which is generated by the set $\{ \rho_i \mid 1 \leq i \leq m \}$.
\end{lemma}

\begin{proof}
This is straightforward from Lemma \ref{dirdecbasic} and Lemma \ref{finitecenter}.
\end{proof}

\begin{lemma} \label{sph2reflectionlemma}
Let $(W,S)$ be a spherical Coxeter system and let $a \neq b \in S^W$
be two reflections of $(W,S)$. Then there exist an element
$w \in W$ and a subset $J$ of $S$ such that $|J| = 2$ and such
that $\langle a,b \rangle^w \leq \langle J \rangle$.
Moreover, if $ab = ba$, then $J$ is of $(-1)$-type and $(ab)^w$ is
the longest element in $(\langle J \rangle,J)$.
\end{lemma}

\begin{proof} This is seen from the
geometric representation of $(W,S)$. We omit the details.
\end{proof}

\subsection*{Parabolic subgroups, finite subgroups and involutions}

\begin{lemma} \label{FormerLemma7}
Let $(W,S)$ be a Coxeter system and let $I,J$ be finite
subsets of $S$. Suppose that there exists an element $w \in W$ such that
$w^{-1} \langle I \rangle w = \langle J \rangle$. Then there
exists an element $v \in \langle I \rangle w$ such that $v^{-1} I v = J$.
In particular, the Coxeter systems $(\langle I \rangle, I)$ and $(\langle J \rangle, J)$
are isomorphic and we have
$|I| = |J|$.
\end{lemma}

\begin{proof}
This follows from Proposition 12 in \cite{FHM}
for Coxeter systems of finite rank. If $(W,S)$ has infinite rank,
we can find a finite subset $K$ of $S$ such that $I \cup J \subseteq K$
and such that $w \in \langle K \rangle$. Thus we can reduce the infinite
rank case to the finite rank case by arguing in
the Coxeter system $(\langle K \rangle,K)$.
\end{proof}

Let $(W,S)$ be a Coxeter system and let $J$ be a spherical subset of $S$.
We denote the longest element of $(\langle J \rangle,J)$ by $\rho_J$.

\begin{lemma} \label{FHMLemmas20+21}
Let $(W,S)$ be a Coxeter system,  let $J \subseteq S$ be finite
and suppose that
$I \subseteq S$ is of $(-1)$-type. Then the following hold.
\begin{itemize}
\item[(i)] If $w \in W$ is such that $w^{-1} \rho_Iw \in \langle J \rangle$,
then $w^{-1}\langle I \rangle w \leq \langle J \rangle$.
\item[(ii)] If $J$ is of $(-1)$-type, then $\{ w \in W \mid w^{-1} \rho_I w = \rho_J \}
= \{ w \in W \mid w^{-1} \langle I \rangle w = \langle J \rangle \}$.
\end{itemize}
\end{lemma}

\begin{proof} In the finite rank case, Assertion (i) is Lemma 20
and Assertion (ii) is Lemma 21 in \cite{FHM}.
As in the proof of the previous lemma, the infinite rank case
is reduced to the finite rank case by considering a finite subset $K$ of $S$
such that $I \cup J \cup \{ w \} \subseteq \langle K \rangle$.
\end{proof}

\smallskip
\noindent
{\bf Definition:}
Let $(W,S)$ be a Coxeter system. A subgroup $P \leq W$ is called a {\sl parabolic subgroup} of $(W,S)$ if
there exist $J \subseteq S$ and $w \in W$ such that $P = \langle J \rangle^w$.
The {\sl parabolic closure} of a subset $X \subseteq W$ in $(W,S)$ is the
intersection of all parabolic subgroups of $(W,S)$ containing $X$. It
is denoted by $Pc_S(X)$.

\medskip
\noindent
{\bf Remark:} The notion of the parabolic closure was introduced by Krammer in \cite{Krammer}.

\smallskip
\noindent
\begin{prop} \label{parabolicfiniterank}
Let $(W,S)$ be a Coxeter system of finite rank, $K \subseteq S$ and $X \subseteq W$.
Then the following hold:
\begin{itemize}
\item[(i)] The parabolic closure $Pc_S(X)$ is a parabolic subgroup of $(W,S)$;
\item[(ii)] if $P$ is a parabolic subgroup of $(W,S)$, then $P \cap \langle K \rangle$
is a parabolic subgroup of $(\langle K \rangle,K)$;
\item[(iii)] if $X \subseteq \langle K \rangle$ and $Pc_K(X)$ denotes the
parabolic
closure in $(\langle K \rangle,K)$, then $Pc_K(X) = Pc_S(X)$.
\end{itemize}
\end{prop}

\begin{proof} Assertion (i) is an immediate consequence of the discussion
following Proposition 2.1.4  in \cite{Krammer}; for a formal statement and proof we refer to Theorem 1.2 in \cite{Qi}. Assertion (ii) follows from
Corollary 7 in \cite{FHAMS} and Assertion (iii) is an immediate consequence of Assertion (ii).
\end{proof}

\begin{cor} \label{parabolicclosurecorollary}
Let $(W,S)$ be a Coxeter system and let $X \subseteq W$ be finite.
Then $Pc_S(X)$ is a parabolic subgroup of $(W,S)$. In particular,
if $U \leq W$ is a finitely generated subgroup, then $Pc_S(U)$ is a parabolic subgroup
of $(W,S)$.
\end{cor}

\begin{proof} As $X$ is a finite subset of $W$, there exists a finite subset $K$ of
$S$ such that $X \subseteq \langle K \rangle$. We denote its parabolic closure
in $(\langle K \rangle,K)$ by $Pc_K(X)$. It follows from Assertion (ii) of
Proposition \ref{parabolicfiniterank} that $Pc_K(X)$ is a parabolic subgroup of
$(\langle K \rangle,K)$. Note that, by definition, each parabolic subgroup of $(\langle K \rangle,K)$
is a parabolic subgroup of $(W,S)$.
Therefore, $Pc_K(X)$ is also a parabolic subgroup of $(W,S)$.
Thus, the corollary is proved if we show $Pc_K(X) = Pc_S(X)$.
As each parabolic subgroup of $(\langle K \rangle,K)$
is a parabolic subgroup of $(W,S)$ we have $Pc_S(X) \subseteq Pc_K(X)$.
Thus, the corollary is proved if we show $Pc_K(X) \subseteq Pc_S(X)$
which in turn follows from the following.

\smallskip
\noindent
{\bf Claim:} If $P$ is a parabolic subgroup of $(W,S)$ containing $X$, then it contains
$Pc_K(X)$ as well.

\smallskip
\noindent
{\sl Proof of the Claim:} Let $P$ be a parabolic subgroup of $(W,S)$ which contains $X$.
By definition there exist $J \subseteq S$ and $w \in W$ such that $P = w \langle J \rangle w^{-1}$.
As $X \subseteq P$ we have $Y := w^{-1} X w \subseteq \langle J \rangle$. As $X$ is finite,
the set $Y$ is finite as well and therefore there exists a finite subset $J_1$ of $J$
such that $Y \subseteq \langle J_1 \rangle$. We have also a finite subset $M$ of $S$
such that $w \in \langle M \rangle$. Now $L:= K \cup J_1 \cup M$
is a finite subset of $S$. Let $Pc_L(X)$ denote the parabolic closure of $X$ in $(\langle L \rangle, L)$.
Since $K \subseteq L$ and $X \subseteq \langle K \rangle$, it follows by Assertion (iii)
of Proposition \ref{parabolicclosurecorollary} (with $(W,S) = (\langle L \rangle, L)$)
that $Pc_K(X) = Pc_L(X)$. As $J_1 \subseteq L$ and $w \in \langle L \rangle$,
the group $P_1 := w \langle J_1 \rangle w^{-1}$ is a parabolic subgroup of $(\langle L \rangle, L)$
which contains $X$. Thus it follows that $Pc_K(X) = Pc_L(X) \leq P_1 \leq P$
which yields the claim.
\end{proof}

\medskip
\noindent
{\bf Remark:} In a Coxeter system $(W,S)$ of arbitrary rank it is not true,
that the parabolic closure of a subset $X$ is a parabolic subgroup.
A counter example is described in the introduction of \cite{NuidaJAlg},
where the problem of sensibly generalizing
the concept of the parabolic closure
to Coxeter systems of arbitrary rank is addressed.

\begin{lemma} \label{Bourbakiexercise}
Let $(W,S)$ be a Coxeter system and let $U \leq W$ be a finite
subgroup of $W$. Then there exist a spherical subset $J$ of $S$
and an element $w \in W$ such that $U^w := w^{-1}Uw \leq \langle J \rangle$.
In particular, the parabolic closure $Pc_S(U)$ is a finite group.
\end{lemma}

\begin{proof}
This follows from Assertion (a) of Proposition 3.2.1 in \cite{Krammer} in the finite
rank case. The reduction of the general case to the finite rank case is straightforward.
\end{proof}

\begin{cor} \label{finitesubgroupcor}
Let $(W,S)$ be a Coxeter system, let $U \leq W$ be a subgroup
and let $x \in W$ be such that $\langle U,x \rangle$ is a finite
subgroup of $W$. Then $\langle Pc_S(U),x \rangle$ is also a finite
subgroup of $W$.
\end{cor}

\begin{proof}
Let $H := \langle U,x \rangle$. By Lemma \ref{Bourbakiexercise} there exist
a spherical subset $J \subseteq S$ and $w \in W$ such that $H^w \leq \langle J \rangle$.
We have $U^w \leq \langle J \rangle$ and therefore $Pc_S(U^w) \leq \langle J \rangle$.
Hence $(\langle Pc_S(U),x \rangle)^w = \langle Pc_S(U^w),x^w \rangle \leq \langle J \rangle$
is a finite group which yields the claim.
\end{proof}

\begin{lemma} \label{2reflectionlemma}
Let $(W,S)$ be a Coxeter system and let $a \neq b \in S^W$ be
such that $ab$ has finite order. Then there exist an element $w \in W$
and a subset $J$ of $S$ such that $|J| = 2$ and  $\langle a,b \rangle^w \leq \langle J \rangle$.
Moreover, if $ab = ba$, then $J$ is of $(-1)$-type and $(ab)^w$ is
the longest element in $(\langle J \rangle,J)$.
\end{lemma}

\begin{proof} As $ab$ has finite order, the group $\langle a,b \rangle$ is
finite. By
Lemma \ref{Bourbakiexercise} there exist
a spherical
subset $K$ of $S$ and an element $w \in W$ such that
$\langle a,b \rangle^w \leq \langle K \rangle$.
By Lemma \ref{Coxeterbasics} $a^w,b^w$ are reflections
of $(\langle K \rangle,K)$ and we can apply Lemma \ref{sph2reflectionlemma}
to get the result.
\end{proof}

\begin{cor} \label{FHMLemma20cor}
Let $(W,S)$ be a Coxeter system, let $I \subseteq S$ be of $(-1)$-type
and let $r$ be the longest element in $(\langle I \rangle, I)$.
If $v \in W$ is such that $\langle v,r \rangle$ is a finite subgroup of $W$,
then $\langle \{ v \} \cup I \rangle$ is a finite subgroup of $W$ as well.
\end{cor}

\begin{proof}
Since $\langle v,r \rangle$ is a finite subgroup of $W$,
there exist a spherical subset $J$ of $S$ and $w \in W$ such that
$\langle v,r \rangle^w \leq \langle J \rangle$.
It follows from Lemma \ref{FHMLemmas20+21} that $I^w \subseteq \langle J \rangle$.
As $\langle J \rangle$ is a finite group containing $v^w$ and $I^w$,
the group $\langle \{ v \} \cup I \rangle$ is finite as well.
\end{proof}

\begin{lemma} \label{Richardsonlemma}
Let $(W,S)$ be a Coxeter system and let $r \in W$ be an involution.
Then there exist
a subset $J$ of $S$ and an element $w \in W$
such that $J$ is of $(-1)$-type and such that $w^{-1} r w = \rho_J$.
Moreover, if $K \subseteq S$ is of $(-1)$-type and $v \in W$
is such that $v^{-1} r v = \rho_K$, then $|K| = |J|$.
\end{lemma}

\begin{proof} The first assertion is a well known result of Richardson.
For Coxeter systems of finite rank it can be found, for instance, in Paragraph 8.3 in \cite{Humphreys}
and the reduction of the general case to finite rank causes no problems.
The second
assertion follows from Assertion (ii) of
Lemma \ref{FHMLemmas20+21} and Lemma
\ref{FormerLemma7}.

\end{proof}

Let $(W,S)$ be a Coxeter system and let $r \in W$ be an involution. By the
previous lemma there exists
$J \subseteq S$
of $(-1)$-type such that $r$ is conjugate to $\rho_J$.
The {\sl $S$-rank} of $r$ is defined to be the cardinality of $J$ which
makes sense in view of the second assertion of the previous lemma. Note
that an involution has $S$-rank one if and only if it is a reflection of $(W,S)$.

\begin{prop} \label{normalizerofparabolicprop}
Let $(W,S)$ be a Coxeter system and let $J \subseteq S$ be a spherical subset of $S$
and let $r$ be the longest element of $(\langle J \rangle, J)$.
Then the following hold:

\begin{itemize}
\item[(i)] If $J$ is of $(-1)$-type, then $C_W(r) = N_W(\langle J \rangle)$.
\item[(ii)] Suppose that $[s,J] = 1$ for all $s \in S \setminus J$ having the property
that $\langle \{ s \} \cup J \rangle$ is a finite group. Then $N_W(\langle J \rangle) =
\langle J \rangle \times \langle J^{\perp} \rangle$.
\end{itemize}
\end{prop}

\begin{proof}
Assertion (i) is a consequence of Assertion (ii) of Lemma \ref{FHMLemmas20+21}.
Assertion (ii) is a consequence of the main result in \cite{BrinkHowlett} in the finite rank case.
As in the proof of Lemma \ref{FormerLemma7} the infinite rank case is easily settled using
the fact that the assertion holds in the finite rank case.
\end{proof}

\section{Coxeter generating sets}

Let $(W,S)$ be a Coxeter system. In this section we shall investigate
conditions on the diagram of $(W,S)$ which ensure that the
Coxeter generating set $S$ can be replaced by another one. In the first
paragraph we provide some very basic observations in the situation
where the Coxeter system admits a visible decomposition as a direct product or
as free product with amalgamation. In the second paragraph we provide specific
information in the spherical case which will play an important role in the sequel.

\subsection*{Direct products and free products with amalgamation}

We recall some definitions from basic group theory.
Let $G$ be a group and let $A,B$ be subgroups. We call $G$ the {\sl
direct product of $A$ and $B$} if $A$ and $B$ are normal subgroups, $G = AB$ and $|A \cap B| = 1$.
We call $G$ the {\sl free product with amalgamation of $A$ and $B$} if $G = \langle A,B \rangle$
and if the following universal property is satisfied:

\begin{itemize}
\item[(FPWA)] If $H$ is a group and $\varphi_A:A \rightarrow H$, $\varphi_B:B \rightarrow H$
are homomorphisms such that $\varphi_A \mid_{A \cap B} = \varphi_B \mid_{A \cap B}$, then
there exists a homomorphism $\varphi:G \rightarrow H$ such that $\varphi \mid_A = \varphi_A$
and $\varphi \mid_B = \varphi_B$.
\end{itemize}

We shall need the following basic observation about free products
with amalgamation.

\begin{lemma} \label{FPAbasicobservation}
Let $G$ be a group and let $A,B$ be subgroups of $G$ such that $G$ is the free product
with amalgamation of $A$ and $B$.
Then $ab$ has infinite
order for all $(a,b) \in (A \setminus B) \times (B \setminus A)$.
\end{lemma}

\begin{proof} This can be seen by using the normal form for elements
in free products with amalgamation. See for instance Theorem 1 in Paragraph 1.2 in \cite{tree}.
\end{proof}

\begin{lemma} \label{productlemma}
Let $(W,S)$ be a Coxeter system and let $K,L \subseteq S$ be such that $S = K \cup L$.
\begin{itemize}
\item[(i)] If $K \cap L = \emptyset$ and $m_{kl} = 2$ for all $(k,l) \in K \times L$,
then $W$ is the direct product of $\langle K \rangle$ and $\langle L \rangle$.
\item[(ii)] If $m_{kl} = \infty$ for all $(k,l) \in (K\setminus L) \times (L\setminus K)$,
then $W$ is the free product with amalgamation of $\langle K \rangle$ and $\langle L \rangle$.
\end{itemize}
\end{lemma}

\begin{proof} Note first that $W$ is generated by $\langle K \rangle$ and $\langle L \rangle$
by the assumption that $S = K \cup L$.
Let $J = K \cap L$. Then $\langle J \rangle = \langle K \rangle \cap \langle L \rangle$
by Lemma \ref{Coxeterbasics}.

Under the assumptions of Assertion (i) it follows
$|\langle K \rangle \cap \langle L \rangle| = 1$ and that $\langle K \rangle$ and $\langle L \rangle$
centralize each other. As we already observed that $W$ is generated by $\langle K \rangle$ and
$\langle L \rangle$ the first
assertion follows.

Let $H$ be a group and let $\varphi_K: \langle K \rangle \rightarrow H$ and
$\varphi_L: \langle L \rangle \rightarrow H$ be homomorphisms which coincide
on $\langle K \rangle \cap \langle L \rangle$. Then $\varphi_K$ and $\varphi_L$
coincide on $J$ and therefore we have a mapping $\varphi: S \rightarrow H$
such that $\varphi \mid_K = \varphi_K$ and $\varphi \mid_L = \varphi_L$.
As $\varphi_K$ (resp. $\varphi_L$) is a homomorphism, it follows
that $(\varphi(a)\varphi(b))^{m_{ab}} = 1$ for all $a,b \in K$ (resp. for
all $a,b \in L$). As $m_{kl} = \infty$ for all $(k,l) \in (K\setminus L) \times (L\setminus K)$,
it follows from the universal property of Coxeter systems that $\varphi$
is the restriction of a homomorphism $\varphi_S$ from $W$ to $H$.
Thus we have shown that (FPWA) holds which finishes the proof of Assertion (ii).
\end{proof}

\begin{prop} \label{productprop}
Let $G$ be a group, let $A,B$ be subgroups and let $K$ (resp. $L$) be a Coxeter generating set
of $A$ (resp. $B$).

\begin{itemize}
\item[(i)] If $G$ is the direct product of $A$ and $B$, then $S := K \cup L$ is a Coxeter generating set
of $G$.
\item[(ii)] If $G$ is the free product with amalgamation of $A$ and $B$ and $J$
is a Coxeter generating set of $A \cap B$ contained in $K$ and $L$, then $S := K \cup L$ is
a Coxeter generating set of $G$.
\end{itemize}
\end{prop}

\begin{proof}
Assertion (i) is obvious. In order to prove Assertion (ii) we first remark that
$ab$ has infinite order for all $(a,b) \in (A \setminus B) \times (B \setminus A)$
by Lemma \ref{FPAbasicobservation}.

In view of Lemma \ref{Coxeterbasics} we have $k \not \in \langle J \rangle$
for each $k \in K \setminus L$ and $l \not \in \langle J \rangle$ for each
$l \in L \setminus K$. Thus
$kl$ has infinite order for all $(k,l) \in (K \setminus L) \times (L \setminus K)$.

For $s,t \in S$ we let $m_{st}$ denote the order of $st$. By what we have said so far
we already know that $m_{kl} = \infty$ for all
$(k,l) \in (K \setminus L) \times (L \setminus K)$.

Let $H$ be a group and let  $\alpha: S \rightarrow H$ be a map such that
$(\alpha(x)\alpha(y))^{m_{xy}} = 1$ for all $x,y \in S$ with $m_{xy} \neq \infty$.
As $K$ (resp. $L$) is a Coxeter generating set of $A$ (resp. $B$),
there exists a unique homomorphism $\varphi_A$ (resp. $\varphi_B$)
from $A$ (resp. $B$) to $H$ such that $\varphi_A(k) = \alpha(k)$ for all $k \in K$
(resp. $\varphi_B(l) = \alpha(l)$ for all $l \in L$).
As $\varphi_A(j) = \alpha(j) = \varphi_B(j)$ for each $j \in J$
and as $\langle J \rangle = A \cap B$, it follows that
$\varphi_A \mid_{A \cap B} = \varphi_B \mid_{A \cap B}$. As $G$ is the free
product with amalgamation of $A$ and $B$ we have a homomorphism
$\varphi:G \rightarrow H$ such that $\varphi \mid_A = \varphi_A$
and $\varphi \mid_B = \varphi_B$ and hence $\varphi \mid_S = \alpha$.
This shows that $(G,S)$ is indeed a Coxeter system and finishes the
proof of Assertion (ii).
\end{proof}

\begin{cor} \label{productcor}
Let $(W,S)$ be a Coxeter system, let $K,L \subseteq S$ be such that $S = K \cup L$ and let
$K'$ (resp. $L'$) be a Coxeter generating set of $\langle K \rangle$ (resp.
$\langle L \rangle$). Then $S' := K' \cup L'$ is a Coxeter generating set of $W$
if one of the following conditions is satisfied.
\begin{itemize}
\item[(i)] $K \cap L = \emptyset$ and $m_{kl} = 2$ for all $(k,l) \in K \times L$;
\item[(ii)] $K \cap L \subseteq K' \cap L'$ and $m_{kl} = \infty$ for all $(k,l) \in (K \setminus L)
\times (L \setminus K)$.
\end{itemize}
\end{cor}

\begin{proof}
Let $A := \langle K \rangle$ and $B := \langle L \rangle$.

Under the assumptions of Assertion (i) it follows from Lemma \ref{productlemma} that
$W$ is the direct product of $A$ and $B$ and the claim follows then from
the first assertion of Proposition \ref{productprop}.

Under the assumptions of Assertion (ii) it follows from Lemma \ref{productlemma}
that $W$ is the free product with amalgamation of $A$ and $B$. By Lemma \ref{Coxeterbasics}
we also know that $\langle K \cap L \rangle = A \cap B$. Hence the second assertion
of Proposition \ref{productprop} finishes the proof of the second statement of
the corollary.
\end{proof}

\begin{lemma} \label{CapraceAppendix}
Let $(W,S)$ be a Coxeter system and let $R$ be a Coxeter generating set of $W$.
If $R \subseteq S^W$, then $R^W = S^W$.
\end{lemma}

\begin{proof}
This is a consequence of a result on subgroups of Coxeter groups
generated by reflections which has been found independently by Deodhar \cite{Deodhar} and
Dyer \cite{Dyer}. For an explicit statement and proof we refer to Corollary A.2
in \cite{CP}.
\end{proof}

\subsection*{Spherical Coxeter systems}

\begin{lemma} \label{blowingdownlemma}
Let $(W,S)$ be a spherical Coxeter system, let $s \in S$ be
such that $s \in Z(W)$, let $C := S \setminus \{ s\}$ and
let $\rho$ be the longest element of $(\langle C \rangle,C)$.
Then the following hold for all $1 \leq k \in {\bf N}$:
\begin{itemize}
\item[(i)] Suppose that $(\langle C \rangle,C)$ is of type $I_2(2k+1)$
and that $a \in C$. Then $b := \rho a \rho \in C$ and $(S \setminus \{ s,b \}) \cup \{ s\rho \}$
is a Coxeter generating set of type $I_2(4k+2)$ of $W$.
Moreover, if $R$ is any Coxeter generating set of type $I_2(4k+2)$
of $W$, then $\{ a,b,s\rho \} \subseteq R^W \subseteq a^W \cup (s\rho)^W$.
\item[(ii)] Suppose that $(\langle C \rangle,C)$ is of type $D_{2k+1}$
and that $a \in C$ is such that $b := \rho a \rho \neq a$.
Then $b \in C$ and $(S \setminus \{ s,b \}) \cup \{ s\rho \}$ is a Coxeter generating
set of type $C_{2k+1}$ of $W$. Moreover, if $R$ is any
Coxeter generating set of type $C_{2k+1}$ of $W$, then
$s \rho \in R^W$ and $R^W \cap \{ c,sc \} \neq \emptyset$ for each $c \in C$;
if $a \in R^W$
we have $R \subseteq (s\rho)^W \cup a^W$ and if $sa \in R^W$ we have
$R \subseteq (s\rho)^W \cup (sa)^W$.
\end{itemize}
\end{lemma}

\begin{proof}
As $s$ is in the center of $W$, we have $(as\rho)^2 = a \rho a \rho = ab$
and hence the order of $as\rho$ is $2 m_{ab}$. It readily follows
that $R$ satisfies the Coxeter relations in (i) and (ii). By checking
the orders of the corresponding finite Coxeter groups, it follows
that $R$ is indeed a Coxeter generating set. The second statement in
Assertion (i) follows from the fact that all non-central involutions are
reflections in a dihedral group. The second statement in Assertion (ii)
follows by considering the automorphism group of the Coxeter groups
of type $C_{2k+1}$ (see for instance Theorem 31 in \cite{FH}).
\end{proof}

\begin{prop} \label{parisprop}
Let $(W,S)$ be an irreducible spherical Coxeter system.
Suppose that there are subgroups $A$ and $B$ such that
$1 < |A| \leq |B|$ and such that $W$ is the direct product of $A$ and $B$.
Then $A = Z(W)$ and one of the following holds:
\begin{itemize}
\item[(i)] $(W,S)$ is of type $E_7$ and $H_3$ and $B$ is not
a Coxeter group.
\item[(ii)] $(W,S)$ is of type $I_2(4k+2)$
for some $1 \leq k \in {\bf N}$ and $B$ is a Coxeter group. Moreover, all Coxeter generating sets
of $B$ are of type $I_2(2k+1)$.
\item[(iii)] $(W,S)$ is of type $C_{2k+1}$ for
some $1 \leq k \in {\bf N}$ and $B$ is a Coxeter group. Moreover, all Coxeter generating sets of $B$
are of type $D_{2k+1}$.
\end{itemize}
\end{prop}

\begin{proof}
Let $(W,S)$ be an irreducible spherical Coxeter system. Then the decompositions
of $W$ as a non-trivial direct product can be determined by going through the list.
This is done in the final section of  \cite{Pa07}. It follows from
Fact 1 of the introduction that all irreducible Coxeter generating sets of
a finite Coxeter group are conjugate in its automorphism group. This observation yields
the second statements of Assertions (ii)
and (iii).
\end{proof}

\begin{cor} \label{decompositionsummary}
Let $(W,R)$ be an irreducible spherical Coxeter system and let $W = A \times B$
be a non-trivial decomposition of $W$ as a direct product such that $|A| \leq |B|$.
Then $(W,R)$ is of $(-1)$-type and  $A = \langle s \rangle$ where $s$ is the longest
element in $(W,R)$ and $B$ admits no proper decomposition as a direct product.

If $B$ is a Coxeter group with Coxeter generating set $C$
and if $\rho$ denotes the
longest element in $(B,C)$, then $s\rho \in R^W$. Moreover, one of the following
holds:

\begin{itemize}
\item[$I$)] There exists $1 \leq k \in {\bf N}$ such that $(W,R)$
is of type $I_2(4k+2)$ and $(B,C)$ is of type $I_2(2k+1)$.
Moreover $C \subseteq R^W$ and $R \subseteq (s\rho)^W \cup a^W$
for any $a \in C$.
\item[$D$)] There exists $1 \leq k \in {\bf N}$ such that $(W,R)$
is of type $C_{2k+1}$ and $(B,C)$ is of type $D_{2k+1}$.
Moreover $C \subseteq R^W$ and $R \subseteq (s\rho)^W \cup a^W$
for any $a \in C$.
\item[$\bar{D}$)] There exists $1 \leq k \in {\bf N}$ such that $(W,R)$
is of type $C_{2k+1}$ and $(B,C)$ is of type $D_{2k+1}$.
Moreover $\bar{C} := \{ sc \mid c \in C \} \subseteq R^W$ and $R \subseteq (s\rho)^W \cup (sa)^W$
for any $a \in C$.
\end{itemize}
\end{cor}

\begin{proof}
This follows from Lemma \ref{blowingdownlemma} and Proposition \ref{parisprop}.
\end{proof}

\section{Right-angled generators}

\label{rightangledsec}

In this section we prove basic facts about right-angled generators of Coxeter systems.
In the finite rank case, several of these facts follow from the more general results about
the finite continuation in \cite{FHM}. Although the results there do not apply directly to the infinite
rank case, our arguments are nevertheless inspired by that paper. Combining our results
on right-angled generators with the information of the previous sections we will accomplish
the 'easy' direction of the main result which is Proposition \ref{mainresulteasydirectionprop}.
Furthermore, we shall give a brief outline on the proof of the opposite direction of the main result.

\smallskip
\noindent
{\bf Definition:} Let $(W,S)$ be a Coxeter system. For $s \in S$ we put
$s^{\perp} := \{ t \in S \mid m_{st} =2 \}$ and $s^{\infty} := \{ t \in S \mid m_{st} = \infty \}$
and we call $s$ a {\sl right-angled generator of $(W,S)$} if $S = \{ s \} \cup s^{\perp} \cup s^{\infty}$.

\smallskip
\noindent
{\bf Convention for this section:} Throughout this section $(W,S)$ is
a Coxeter system (possibly of infinite rank) and $s \in S$ is a right-angled
generator of $(W,S)$. Moreover, $\pi:W \longrightarrow \{ +1,-1 \}$
is the unique homomorphism mapping $s$ onto $-1$ and $s'$ onto $+1$
for all $s \neq s' \in S$ and $V$ denotes the kernel of $\pi$.

\smallskip

\begin{lemma} \label{rightangledcentralizer}
The centralizer of $s$ in $W$ is the group $\langle \{ s \} \cup s^{\perp}\rangle$.
\end{lemma}

\begin{proof} This follows from Assertion (ii) of Proposition
\ref{normalizerofparabolicprop} with $J := \{ s \}$.
\end{proof}

\smallskip
\noindent
\begin{lemma} \label{rightangled1}
We have $s^W \cap \langle \{ s \} \cup s^{\perp} \rangle = \{ s \}$.
\end{lemma}

\begin{proof} Let $w \in W$ be such that $t =s^w \in \langle \{ s \} \cup s^{\perp} \rangle$.
As $s$ is not in $V$ which is a normal subgroup of $W$, it follows that
$t$ is not in $V$. As $t \in \langle \{ s \} \cup s^{\perp} \rangle$, there
exists an element $u \in \langle s^{\perp} \rangle \leq C_W(s)$ such that $t = su$.
As $t^2 = s^2 = [s,u] = 1$, we have $u^2 = 1$.

Suppose $u \neq 1_W$. By Lemma \ref{Richardsonlemma}
applied to $(\langle s^{\perp} \rangle,s^{\perp})$ there exists a non-empty $(-1)$-set
$J \subseteq s^{\perp}$
and an element $x \in \langle s^{\perp} \rangle \leq C_W(s)$ such that $u^x$ is the longest
element in $(\langle J \rangle, J)$. As $[s,x] = 1$ it follows
that $(s^w)^x = (su)^x = su^x$ is the longest element of
the $(-1)$-set $K:= \{ s \} \cup J$ which is a contradiction, because
$|K| \geq 2$ and $s$ is a reflection of $(W,S)$.
\end{proof}

\begin{lemma} \label{rightangled2}
Let $x \in W$ be such that $\langle s,x \rangle$ is a finite subgroup
of $W$. Then $[s,x]=1$ and in particular $x \in \langle \{ s \} \cup s^{\perp} \rangle$.
\end{lemma}

\begin{proof} We first consider the case where $x \in V$.
By Lemma \ref{Bourbakiexercise} there exist a spherical subset $J$ of $S$ and
an element $w \in W$ such that $\langle s,x \rangle^w \leq \langle J \rangle$.
As $\langle s,x \rangle$ is not contained in $V$ which is a normal subgroup
of $W$, $\langle s,x \rangle^w$ is not contained in $V$. As $S \setminus \{ s \} \subseteq V$
it follows that $s \in J$. Since $J$ is spherical and contains $s$,
it follows that $J \subseteq \{ s \} \cup s^{\perp}$. It follows
that $s^w \in \langle J \rangle \leq \langle \{ s \} \cup s^{\perp} \rangle$
and applying Lemma \ref{rightangled1} we see that $s^w = s$ which means $[s,w] = 1$.
Note that $x^w \in V \cap \langle \{ s \} \cup s^{\perp} \rangle = \langle s^{\perp} \rangle$
and hence $[s,x]^w = [s^w,x^w] = [s,x^w] = 1$ which implies $[s,x] = 1$.

Suppose now that $x$ is not in $V$. Then $v := sx \in V$ and $\langle s,v \rangle = \langle s,x \rangle$
is a finite subgroup of $W$. By what we know already it follows that $1 = [s,v] = [s,sx]$
which implies $[s,x] = 1$.

The last assertion of the lemma follows from Lemma \ref{rightangledcentralizer}.
\end{proof}

\begin{prop} \label{rightangled3}
Let $r \in  \langle s^{\perp} \rangle$ be an involution and $x \in W$ be such that
$\langle sr,x \rangle$ is a finite subgroup of $W$. Then $[s,x] =1$ and
in particular $x \in \langle \{ s \} \cup s^{\perp} \rangle$.
\end{prop}

\begin{proof} As in the proof of Lemma \ref{rightangled2} we consider first the case where $x \in V$.
As $r \in \langle s^{\perp} \rangle$ there exists a $(-1)$-set $J \subseteq s^{\perp}$
and $y \in \langle s^{\perp} \rangle$ such that $u := r^y$ is the longest
element in $(\langle J \rangle,J)$. Setting $K := J \cup \{ s \}$,
it follows that $K$ is of $(-1)$-type and that $su$ is the longest element in $(\langle K \rangle, K)$.
As $y \in \langle s^{\perp} \rangle$ it follows that $s^y = s$ and therefore
$\langle su, x^y \rangle$ is a finite group. It follows from Corollary \ref{FHMLemma20cor}
that $\langle K \cup \{ x^y \} \rangle$ is a finite group which implies
that $\langle s,x^y \rangle$ is a finite group.
Applying now Lemma \ref{rightangled2} we
have $1 = [s,x^y] = [s^y,x^y] = [s,x]^y$ and hence $[s,x] = 1$.

Suppose now that $x$ is not in $V$ and put $v := sr x$.
Then $v \in V$ and $\langle sr, v \rangle = \langle sr, x \rangle$
is a finite group. By what we know from the first case, it follows
that $1 = [s,v] = [s,sr x]$ and as $[s,r]=1$ is follows that $[s,x] =1$.

The last assertion of the lemma follows from Lemma \ref{rightangledcentralizer}.
\end{proof}

\begin{cor} \label{infordercor}
Let $r \in \langle s^{\perp} \rangle$ be an involution and let $u \in s^{\infty}$.
Then $(sr)u$ has infinite order.
\end{cor}

\begin{proof}
Assume by contradiction that $(sr)u$ has finite order.
As $sr$ and $u$ are both involutions, it follows that $\langle sr,u \rangle$
is a finite subgroup of $W$. Thus we are in the position to apply Proposition
\ref{rightangled3} which yields $[s,u]=1$ and hence a contradiction.
\end{proof}

\subsection*{$s$-components}

\smallskip
\noindent
We first recall the definition of an $s$-component:
An $s$-component is a spherical irreducible component of $(\langle s^{\perp} \rangle,s^{\perp})$.

\begin{prop} \label{s-Translation}
Let $C$ be an $s$-component of $(-1)$-type and let $\rho$ be
the longest element in $(\langle C \rangle,C)$. Then
$R := \{ s\rho \} \cup (S \setminus \{ s \})$ is a Coxeter generating
set of $W$ and $s$ is not a reflection of $(W,R)$.
\end{prop}

\begin{proof}
We define the map $\alpha: S \rightarrow W$ by
$\alpha(s) := s \rho$ and $\alpha(t) := t$ for all $t \in S \setminus \{ s \}$.

We claim that the order of $\alpha(u)\alpha(v)$ is equal to the order of $uv$
for all $u,v \in S$. This is obvious if $u \neq s \neq v$ and if $u = s = v$
since $\rho$ is an involution contained in $\langle s^{\perp} \rangle \leq C_W(s)$
and $\alpha(s) = s\rho$. It remains to consider the case where $u = s \neq v$.
As $C$ is an irreducible component of $(\langle s^{\perp} \rangle, s^{\perp})$
of $(-1)$-type, it follows that $\rho$ is in the center of $\langle s^{\perp} \rangle$.
Thus, if $v \in s^{\perp}$, we have $[\alpha(s),\alpha(v)] = [s\rho,v] =1$
and hence the order of $\alpha(s)\alpha(v) = s \rho v$ is equal 2 which
is also the order of $sv$. If $v \in s^{\infty}$, then $\alpha(s)\alpha(v)=
s\rho v$ has infinite order by Corollary \ref{infordercor}. Thus our claim is proved.

In view of the above and by the universal property of $(W,S)$, there is a unique endomorphism $\tau$
of $W$ such that $\tau \mid_S = \alpha$ and its square is easily seen to be
the identity on $W$. Thus $\tau$ is an automorphism and $R := \tau(S)$ is
a Coxeter generating set of $W$.

Assume, by contradiction, that $s \in R^W$. Then it follows that
$S \subseteq R^W$ because $S \setminus \{ s \} \subseteq R$.
By Lemma \ref{CapraceAppendix} this implies
$S^W = R^W$. As $s \rho \in R$ it follows that $s \rho \in S^W$.
This is a contradiction because $s \rho$ is the longest element
of $(\langle \{ s \} \cup C \rangle, \{ s \} \cup C)$ and $|\{ s \} \cup C| > 1$.
\end{proof}

\begin{cor} \label{s-Transvectioncor}
If there exists an $s$-component of $(-1)$-type, then $s$
is not an intrinsic reflection of $W$.
\end{cor}

\subsection*{Blowing down generators}

We first recall the definition of a blowing down generator for $s$:
We call $a \in s^{\perp}$ a {\sl blowing down generator for $s$} if the following conditions
are satisfied:

\begin{itemize}
\item[(BDG1)] If $C$ denotes the irreducible component of $(\langle s^{\perp} \rangle, s^{\perp})$
containing $a$, then $(\langle C \rangle,C)$ is of type $I_2(2k+1)$ or $D_{2k+1}$
for some $k \in {\bf N}$; moreover, if $\rho$ denotes the longest element in the
Coxeter system $(\langle C \rangle,C)$, then $b:= \rho a \rho \neq a$.
\item[(BDG2)] If $u_0,u_1,\ldots,u_n \in s^{\infty}$ are such that
$u_{i-1}u_i$ has finite order for all $1 \leq i \leq n$, then there exists
an element $x \in \{ a,b \}$ such that $\{ u_i \mid 0 \leq i \leq n \} \subseteq x^{\infty}$.
\end{itemize}

Let $a$ be a blowing down generator for $s$ and let $C,\rho$ and $b = \rho a \rho$
be as in the previous definition. Then we call $a$ a {\sl proper} blowing
down generator for $s$ if $s^{\infty} \subseteq b^{\infty}$.

\begin{prop} \label{diagramtwistprop}
Suppose that $a \in s^{\perp}$ is a blowing down generator for $s$.
Then there exists a Coxeter generating set $S_1$ of $W$ with the following properties.
\begin{itemize}
\item[(i)] $\{ s \} \cup s^{\perp} \subseteq S_1$;
\item[(ii)] $S_1^W = S^W$;
\item[(iii)] $s$ is a right-angled generator of $(W,S_1)$;
\item[(iv)] $a$ is a proper blowing down generator for $s$ with respect to
the generating set $S_1$.
\end{itemize}
\end{prop}

\begin{proof}
Let $K_0 := \{ t \in s^{\infty} \mid m_{bt} \neq \infty \}$ and
for $1 \leq n \in {\bf N}$ let $K_n := \{ u \in s^{\infty} \mid m_{tu} \neq \infty
\mbox{ for some }t \in K_{n-1} \}$ and put $K := \cup_{n \geq 0} K_n$.
It follows from the construction of the set $K$
that $m_{kl} = \infty$ for all $(k,l) \in K \times (s^{\infty} \setminus K)$.
As $a$ is a blowing down generator for $s$ it follows also that $K \subseteq a^{\infty}$.
As $C$ is an irreducible component of $(\langle s^{\perp} \rangle, s^{\perp})$
we have also that $(\{ s \} \cup s^{\perp})\setminus C \subseteq C^{\perp}$.
We are now in the position to apply a diagram twist as described in
\cite{BMMN} Definition 4.4. with $V := C$ and $U:=K$.
Setting $S_1 = (S \setminus K) \cup K'$
with $K' := \{ \rho x \rho \mid x \in K \}$ we obtain a new Coxeter generating
set of $W$ by Theorem 4.5 of \cite{BMMN}.
By the definition of $S_1$ we have $\{ s \} \cup s^{\perp} \subseteq S_1 \subseteq S^W$, which implies
$S^W = S_1^W$ by Lemma \ref{CapraceAppendix}. A straight forward checking reveals
that $s$ is also a right-angled generator of $(W,S_1)$ and that $a$ is indeed
a proper blowing down generator for $s$ with respect to $S_1$.
\end{proof}

\begin{prop} \label{blowingdownprop}
Let $a \in s^{\perp}$ be a proper blowing down generator for $s$.
Let $C$ be the irreducible component of $(\langle s^{\perp} \rangle,s^{\perp})$
containing $a$,
let $\rho$ be the longest element in $(\langle C \rangle, C)$
and let $b := \rho a \rho$.
Then $R:= (S \setminus \{ s,b \}) \cup \{ s\rho \}$
is a Coxeter generating set of $W$ and $s \not \in R^W$.
\end{prop}

\begin{proof}
Let $I := \{ s \} \cup C$ and $I_1 := (I \setminus \{ s,b \}) \cup \{ s\rho \}$.
It follows from Lemma \ref{blowingdownlemma} that $I_1$ is a Coxeter generating set of $\langle I \rangle$.
Let $X := \{ s \} \cup s^{\perp} \setminus I$.
Then $m_{ix} = 2$ for all $(i,x) \in I \times X$.
By Assertion (i) of Corollary \ref{productcor} it follows that
$I_1 \cup X$ is a Coxeter generating set of $\langle \{ s \} \cup s^{\perp} \rangle$.

We now set $K := \{ s \} \cup s^{\perp}$. By what we have just proved it follows
that $K_1 := (K \setminus \{ b,s \}) \cup \{ s\rho \}$ is a Coxeter generating set of $K$.
Setting $L := (K \setminus \{ b,s \}) \cup s^{\infty}$ we have $J := K \cap L = K \setminus \{ s,b \}$.
Thus $K \setminus L = \{ s,b \}$ and $L \setminus K = s^{\infty} \subseteq b^{\infty}$.
We conclude that $m_{kl} = \infty$ for all $(k,l) \in (K \setminus L) \times (L \setminus K)$.
Thus we are in the position to apply Assertion (ii) of Corollary \ref{productcor}
in order to see that $(S \setminus \{ s,b \}) \cup \{ s\rho \}$
is a Coxeter generating set of $W$.

As $a \in R$, we have also $b = \rho a \rho \in R^W$.
Assume, by contradiction, that $s \in R^W$. Then $S \subseteq R^W$
which implies $S^W = R^W$ by Lemma \ref{CapraceAppendix}
and hence $s\rho \in R$ is a reflection of $(W,S)$.
As $\rho \neq 1$ Lemma \ref{rightangled1} yields a contradiction.
Hence $s$ is not a reflection of $(W,R)$ which concludes the proof
of the proposition.
\end{proof}

\begin{cor} \label{blowingdowncor}
If there exists a blowing down generator $a \in s^{\perp}$ for $s$,
then $s$ is not an intrinsic reflection of $(W,S)$.
\end{cor}

\begin{proof}
Let $a \in s^{\perp}$ be a blowing down generator for $s$ and let $S_1$ be as in
Proposition \ref{diagramtwistprop}.
Now, by Proposition \ref{blowingdownprop} applied to $(W,S_1)$ there exists a Coxeter generating set
$R$ of $W$ such that $s$ is not a reflection of $(W,R)$.
\end{proof}

\subsection*{Remark on the proof of the main result}

Corollaries \ref{s-Transvectioncor} and \ref{blowingdowncor}
yield the following Proposition.

\begin{prop} \label{mainresulteasydirectionprop}
If $s$ is an intrinsic reflection of $(W,S)$, then each $s$-component has
trivial center and there are no blowing down generators for $s$.
\end{prop}

Proposition \ref{mainresulteasydirectionprop} provides
one direction of the main result. In the remainder
of the paper we shall prove the Proposition of the introduction which
is just a reformulation of the other direction.
In order to do this
we have to establish the existence of a blowing down generator for $s$. Here,
checking Axiom (BDG2) requires most of the work. In Proposition \ref{reductiontothreecases}
we shall divide up the problem into three cases which will be treated separately
in the subsequent sections. One of these cases is rather easy to handle.
In order to settle the remaining two cases we need an additional tool
which is
Proposition \ref{keyproposition}.
Its proof uses the geometry
of the Cayley graph of $(W,S)$.

\section{Reduction to three cases}

As already pointed out, the proof of the opposite direction of the main result
splits into three cases. In this section we shall establish this case distinction. More specifically, the
goal of this subsection is to prove the following.

\begin{prop} \label{reductiontothreecases}
Let $(W,S)$ be a Coxeter system, let $s \in S$ be a right-angled generator of
$(W,S)$ such that there is no $s$-component of $(-1)$-type and suppose
that there is a Coxeter generating set $T$ of $W$ such that $s$ is not
a reflection of $(W,T)$. Then there exists a Coxeter generating set $R$
of $W$, an irreducible subset $J$ of $R$ of $(-1)$-type and an $s$-component
$C$ such that $s$ is the longest element of $(\langle J \rangle,J)$,
$\langle J \rangle = \langle s \rangle \times \langle C \rangle$
and $s\rho \in J^{\langle J \rangle}$ where $\rho$ denotes the longest
element in $(\langle C \rangle,C)$. Moreover, one of the following holds.

\begin{itemize}
\item[$I$)] There exists $1 \leq k \in {\bf N}$ such that $(\langle J \rangle, J)$
is of type $I_2(4k+2)$ and $(\langle C \rangle,C)$ is of type $I_2(2k+1)$.
Moreover $C \subseteq J^{\langle J \rangle}$ and $J \subseteq (s\rho)^{\langle J \rangle}
\cup a^{\langle J \rangle}$
for any $a \in C$.
\item[$D$)] There exists $1 \leq k \in {\bf N}$ such that $(\langle J \rangle, J)$
is of type $C_{2k+1}$ and $(\langle C \rangle,C)$ is of type $D_{2k+1}$.
Moreover $C \subseteq J^{\langle J \rangle}$ and $R \subseteq (s\rho)^{\langle J \rangle}
\cup a^{\langle J \rangle}$
for any $a \in C$.
\item[$\bar{D}$)] There exists $1 \leq k \in {\bf N}$ such that $(\langle J \rangle, J)$
is of type $C_{2k+1}$ and $(\langle C \rangle,C)$ is of type $D_{2k+1}$.
Moreover $\bar{C} := \{ sc \mid c \in C \} \subseteq J^{\langle J \rangle}$
and $R \subseteq (s\rho)^{\langle J \rangle}
\cup (sa)^{\langle J \rangle}$
for any $a \in C$.
\end{itemize}
\end{prop}

\subsection*{On decompositions of finite Coxeter groups as direct products}

\begin{prop} \label{directproductdecompfinite}
Let $(W,S)$ be a spherical Coxeter system and let $s \in S$ be such that
$Z(W) = \langle s \rangle$. Suppose that $A$ and $B$ are subgroups
of $W$ such that $W = A \times B$. Suppose further that $\langle s \rangle$
is properly contained in $A$ and that there exists a subset $J$ of $A$
such that $(A,J)$ is an irreducible Coxeter system.
Then there exists an irreducible
component $C \neq \{ s \}$ of $(W,S)$
such that $A = \langle s \rangle \times \langle C \rangle$.
\end{prop}

In order to establish the proof of this proposition it is convenient
to recall some basic results on direct product decompositions of
groups and in particular, the Krull-Remak-Schmidt Theorem.

\smallskip
\noindent
{\bf Definition:} A group $G$ is called {\sl indecomposable} if $|G| > 1$ and if $G$ is not
the inner direct product of two non-trivial subgroups.
A family $(H_i)_{1 \leq i \leq n}$ of subgroups of $G$ is called a {\sl Remak-decomposition}
of $G$ if $G$ is the inner direct product of the $(H_i)_{1 \leq i \leq n}$ and if
$H_i$ is indecomposable for $1 \leq i \leq n$.

The following observation is an immediate consequence of Corollary \ref{decompositionsummary}.

\begin{lemma} \label{RemakirredCox}
Let $(W,S)$ be an irreducible spherical Coxeter system and let $H_1,\ldots,H_n$
be non-trivial subgroups of $W$ such that $W = H_1 \times H_2 \ldots \times H_n$.
Then $n \leq 2$.
\end{lemma}

The following is obvious.

\begin{lemma}
Any finite group admits a Remak decomposition.
\end{lemma}

\smallskip
\noindent
{\bf Definition:} Let $G$ be a group and $\alpha \in Aut(G)$. Then $\alpha$
is called a {\sl central automorphism of $G$} if $\alpha$ induces the identity
on $G/Z(G)$, i.e. if $\alpha(g) \in gZ(G)$ for all $g \in G$.

\smallskip
\noindent
We shall need the following version of the Krull-Remak-Schmidt Theorem.

\begin{prop} \label{RKSversion}
Let $G$ be a finite group and let $(H_i)_{1 \leq i \leq n}$
and $(K_j)_{1 \leq j \leq m}$ be Remak decompositions of $G$.
Then $n = m$ and there exists a permutation $\pi \in Sym(n)$
and a central autormorphism $\alpha$ of $G$ such that
$\alpha(H_i) = K_{\pi(i)}$ for all $1 \leq i \leq n$.
\end{prop}

\begin{proof} This is a special case of Theorem 3.3.8 in \cite{Robinson}.
\end{proof}

\begin{cor} \label{KRScorollary}
Let $G$ be a finite group and let $(H_i)_{0 \leq i \leq n}$
be a Remak decomposition of $G$ such that $H_0 = Z(G)$.
Let $A,B \leq G$ be such that $G = A \times B$ and $Z(G) \leq A$.
Then there exist a subset $I$ of $\{ 1,\ldots,n \}$
such that $A = Z(G) \times \langle H_i \mid i \in I \rangle$.
\end{cor}

\begin{proof} Let $(X_i)_{1 \leq i \leq k}$
and $(Y_j)_{1 \leq j \leq l}$ be Remak decompositions of $A$ and $B$.
Then $(X_1,\ldots,X_k,Y_1,\ldots,Y_l)$ is a Remak decomposition of
$G$. As any automorphism of $G$ stabilizes $Z(G) = H_0$ it follows
from Proposition \ref{RKSversion} that $X_i = Z(G)$
for some $1 \leq i \leq k$ because of the assumption $Z(G) \leq A$.
Without loss of generality we may assume $X_k = Z(G)$.
We now put $K_0 = Z(G) = H_0$, $K_i := X_i$ for $1 \leq i \leq k-1$
and $K_i := Y_{i-k+1}$ for $k \leq i \leq m:= l+k-1$.
It follows from Proposition \ref{RKSversion} that $m = n$
and that there are a permutation $\pi \in Sym(n)$ and a central
automorphism $\alpha$ of $G$ such that $\alpha(H_i) = K_i$
for $1 \leq i \leq n$. Setting $I := \pi^{-1}(\{ 1, \ldots, k-1 \})$
the assertion follows.
\end{proof}

\begin{lemma} \label{Remakcenteroforder2}
Let $(W,S)$ be a spherical Coxeter system and let $s \in S$
be such that $Z(W) = \langle s \rangle$. Let $C_1,\ldots,C_n$
be the irreducible components of $(W,S)$ distinct from $C_0 := \{ s \}$.
Then $(\langle C_i \rangle)_{0 \leq i \leq n}$ is a Remak decomposition
of $W$.
\end{lemma}

\begin{proof} Let $1 \leq i \leq n$. The Coxeter system
$(\langle C_i \rangle,C_i)$ is an irreducible Coxeter system.
Moreover, if $z \in Z(\langle C_i \rangle)$, then $z \in Z(W) \cap \langle C_i \rangle$
which is the trivial group. It follows from Corollary \ref{decompositionsummary} that
$\langle C_i \rangle$ is indecomposable which finishes the proof.
\end{proof}



\smallskip
\noindent
{\bf Proof of Proposition \ref{directproductdecompfinite}:}
Let $C_1,\ldots,C_n$
be the irreducible components of $(W,S)$ distinct from $C_0 := \{ s \}$.
By Lemma \ref{Remakcenteroforder2} $(\langle C_i \rangle)_{0 \leq i \leq n}$ is a Remak decomposition
of $W$.

Let $A$ and $B$ be subgroups of $W$ satisfying the hypothesis of the Proposition.
Thus we have that $\langle s \rangle = Z(W)$ is properly contained in $A$.
As $(\langle C_i \rangle)_{0 \leq i \leq n}$ is a Remak decomposition
of $W$ Corollary \ref{KRScorollary} yields a subset
$I$ of $\{ 1,\ldots,n \}$ such that $A = Z(W) \times \langle C_i \mid i \in I \rangle$.
As $Z(W)$ is properly contained in $A$ we have $I \neq \emptyset$.
By our assumption that there is an irreducible Coxeter generating set $J$
of $A$ it follows from Lemma \ref{RemakirredCox} that $|I| \leq 1$.
Thus, $A = Z(W) \times \langle C_i \rangle$ for some $1 \leq i \leq n$
which is precisely the claim of the proposition.

\subsection*{More on right-angled generators}

\smallskip
\noindent
{\bf Convention:} Throughout this subsection $(W,S)$ is a Coxeter system
and $s \in S$ is a right-angled generator of $(W,S)$. We let $\pi:W \rightarrow \{ +1,-1 \}$
be the unique homomorphism mapping $s$ onto $-1$ and $s'$ onto $+1$ for all $s' \in S$
which are distinct from $s$ and we denote its kernel by $V$.
Moreover,
we assume that there is no $s$-component of $(W,S)$ which is
of $(-1)$-type.

\begin{lemma} \label{centerobservation}
The center of $\langle s^{\perp} \rangle$ is trivial.
\end{lemma}

\begin{proof} This follows from the fact, that there is no
irreducible component of $(\langle s^{\perp} \rangle,s^{\perp})$ of
$(-1)$-type.
\end{proof}

\begin{lemma} \label{conjlemma}
Let $T$ be a Coxeter generating set of $W$ such that $s$ is
not a reflection of $(W,T)$. Then there exists a Coxeter generating
set $R$ of $W$ and a $(-1)$-subset $J$ of $R$ such that
$|J| \geq 2$ and $s$ is the longest element of $(\langle J \rangle,J)$.
\end{lemma}

\begin{proof}
By Lemma \ref{Richardsonlemma} there exists a $(-1)$-subset $K$ of $T$ and an element
$w \in W$ such that $s^w$ is the longest element in $(\langle K \rangle,K)$.
As $s$ is not in $T^W$, we have $|K| \geq 2$. Setting $v:= w^{-1}$, $R:=T^v$ and
$J := K^v$ the claim follows.
\end{proof}

\begin{prop} \label{proofofprop}
Let $R$ be a Coxeter generating set of $W$ and let
$J \subseteq R$ be of $(-1)$-type such that
$s$ is the longest element of $(\langle J \rangle, J)$.
Then the following hold:
\begin{itemize}
\item[(i)] If $r \in R \setminus J$ is such that $\langle \{ r \} \cup J \rangle$
is a finite group, then $[J,r] = 1$.
\item[(ii)] Let $K := \{ r \in R \mid r \not \in J \mbox{ and } [r,J] = 1 \}$. Then
$\langle \{ s \} \cup s^{\perp} \rangle =C_W(s) = N_W(\langle J \rangle) = \langle J \rangle \times
\langle K \rangle$.
\item[(iii)] $J$ is an irreducible subset of $R$.
\item[(iv)] if $|J| \geq 2$ then there exists a  spherical irreducible component of
$C$ of $(\langle s^{\perp} \rangle, s^{\perp})$ such that $\langle J \rangle = \langle s \rangle
\times \langle C \rangle$.
\end{itemize}
\end{prop}

\begin{proof} As $\langle J \rangle$ is a finite subgroup
containing $s$ we have $\langle J \rangle \leq \langle \{ s \} \cup s^{\perp} \rangle$
and also that $\langle J \rangle$ is not contained in $V$. Thus there exists
a $\sigma \in J$ such that $\sigma \not \in V$ which we can write as $\sigma = s \rho$
for some involution $\rho \in \langle s^{\perp}\rangle$. Suppose that
$r \in R \setminus J$ is such that $r\sigma$ has finite order,
then $\langle \sigma, r \rangle$ is a finite subgroup of $W$ and by
Proposition \ref{rightangled3} we obtain that $[s,r] = 1$.
As $s$ is the longest element in $(\langle J \rangle,J)$, it follows
that $r \in C_W(s) =  N_W(\langle J \rangle)$. We conclude that $[r,J] = 1$.
This yields Assertion (i).

As $s$ is the longest element of $(\langle J \rangle,J)$
we have $C_W(s) = N_W(\langle J \rangle)$ by Assertion (i) of Proposition \ref{normalizerofparabolicprop};
Assertion (ii) of Proposition \ref{normalizerofparabolicprop} and Assertion (i)
yield $N_W(\langle J \rangle) = \langle J \rangle \times \langle K \rangle$;
finally we have $C_W(s) = \langle \{ s \} \cup s^{\perp} \rangle$ by
Lemma \ref{rightangledcentralizer}. This finishes the proof of Assertion (ii).

By Assertion (ii) we have decomposition of $\langle \{ s \} \cup s^{\perp} \rangle$
as a direct product $\langle J \rangle \times \langle K \rangle$.
Assume by contradiction, that $J$ is not irreducible. Then, by Lemma \ref{longesteltoffinite}, the
center of $\langle J \rangle$ is of order at least 4, because
$(\langle J \rangle,J)$ is of $(-1)$-type. It follows that
the center of $\langle \{ s \} \cup s^{\perp} \rangle$ has order
at least 4. Moding out $\langle s \rangle$ this yields
that $\langle s^{\perp} \rangle$ has a non-trivial center and contradicts Lemma \ref{centerobservation}.
This finishes the proof of Assertion (iii).

Let $P :=Pc_S(\langle J \rangle)$ be the parabolic closure of $\langle J \rangle$ in $(W,S)$.
As $\langle J \rangle$ is a finite normal subgroup of $\langle \{ s \} \cup s^{\perp} \rangle$,
the group $P$ is also a finite, normal subgroup of $\langle \{ s \} \cup s^{\perp} \rangle$.
As $P$ is a parabolic subgroup of the Coxeter system
$(\langle \{ s \} \cup s^{\perp} \rangle, \{ s \} \cup s^{\perp})$, it follows
that there are irreducible spherical components $C_1,\ldots,C_n$ of $(\langle s^{\perp} \rangle, s^{\perp})$
such that $P = \langle s \rangle  \times \langle C_1 \rangle \times \ldots \times \langle C_n \rangle$.
For $1 \leq i \leq n$ the set $C_i$ is an $s$-component and since there
are no $s$-components of $(-1)$-type by assumption, it follows that
the center of $\langle C_i \rangle$ is trivial. Thus
$\langle s \rangle$ is the center of $P$.

Setting $Q := \langle K \rangle \cap P$ we have a direct decomposition
$P = \langle J \rangle \times Q$. We are now in the position
to apply Proposition \ref{directproductdecompfinite} with $W:=P, A:=\langle J \rangle$
and $B:= Q$. It follows that there exists an $1 \leq i \leq n$
such that $\langle J \rangle = \langle s \rangle \times \langle C_i \rangle$.
As $C_i$ is an irreducible spherical component of $(\langle s^{\perp} \rangle,s^{\perp})$
this finishes the proof of Assertion (iv).
\end{proof}

\medskip
\noindent
{\bf Proof of Proposition \ref{reductiontothreecases}:} By Lemma \ref{conjlemma}
there exists a Coxeter generating set $R$ of $W$ and a $(-1)$-subset $J$ of $R$
such that $|J| \geq 2$ and such that $s$ is the longest element of $(\langle J \rangle,J)$.
By Assertion (iv) of Proposition \ref{proofofprop} there exists an irreducible
spherical component $C$ of $(\langle s^{\perp} \rangle,s^{\perp})$
such that $\langle J \rangle = \langle s \rangle \times \langle C \rangle$.
Applying now Corollary \ref{decompositionsummary} with
$(W,R) := (\langle J \rangle,J), A:= \langle s \rangle$ and $B := \langle C \rangle$,
yields the proposition.

\section{The case $\bar{D}$}

\label{sectioncasebarD}

\noindent
{\bf Convention:} Throughout this section $(W,S)$ is a Coxeter system, $s \in S$
is a right-angled reflection of $(W,S)$, $\pi: W \rightarrow \{ +1,-1 \}$
is the homomorphism which sends $s$ onto $-1$ and $t$ onto $+1$ for
all $t \in S \setminus \{ s \}$ and $V \leq W$ is its kernel. Moreover,
$C \subseteq s^{\perp}$ is a $s$-component of type $D_{2k+1}$,
$\rho$ is the longest element of $(\langle C \rangle,C)$ and
$a \in C$ is such that $\rho a \rho \neq a$. Finally, $R \subseteq W$
is a Coxeter generating set of $W$ and $J \subseteq R$ is of type $C_{2k+1}$ and such that
$\langle J \rangle = \langle s \rangle \times \langle C \rangle$ and
$J \subseteq (s\rho)^{\langle J \rangle} \cup (sa)^{\langle J \rangle}$.

\medskip
\noindent
The goal of this section is to prove the following.

\begin{prop} \label{barDprop}
The order of $cu$ is infinite for all $c \in C$ and all $u \in s^{\infty}$.
In particular, $a$ is a blowing down generator for $s$.
\end{prop}

\begin{lemma} \label{ruhasinfiniteorder}
Suppose that $r \in J$ and $x \in W$ are such that
$\langle r,x \rangle$ is a finite subgroup of $W$. Then $x$ normalizes $\langle J \rangle$.
\end{lemma}

\begin{proof}
As $\rho$ and $a$ are in $V$, it follows that $s\rho$ and $sa$ are not in $V$
and therefore $r \in (s\rho)^{\langle J \rangle} \cup (sa)^{\langle J \rangle}$
is not in $V$ because $V$ is a normal subgroup of $W$.
On the other hand, $r \in \langle J \rangle = \langle s \rangle \times \langle C \rangle$ which implies that
there is an involution $\omega \in \langle C \rangle \leq \langle s^{\perp} \rangle$ such
that $r = s\omega$. By our assumption, $\langle s\omega, x \rangle$
is a finite subgroup of $W$. Thus it follows by Proposition \ref{rightangled3} that
$[s,x] =1$ and in particular $x \in \langle \{ s \} \cup s^{\perp} \rangle$.

As $C$ is an irreducible component of $(\langle \{ s \} \cup s^{\perp} \rangle,  \{ s \} \cup s^{\perp})$
it follows that $x$ normalizes $\langle C \rangle$. Now, $x$ centralizes $\langle s \rangle$
and normalizes $\langle C \rangle$. Thus it normalizes $\langle J \rangle = \langle s \rangle \times
\langle C \rangle$ and we are done.
\end{proof}

\begin{lemma} \label{cuhasfiniteorder}
Let $1 \neq y \in \langle J \rangle$ and $x \in W$ be such that
$\langle y, x \rangle$ is a finite group. Then $x \in \langle \{ s \} \cup s^{\perp} \rangle$.
\end{lemma}

\begin{proof}
Let $Y := Pc_R(y)$. As $1 \neq y \in \langle J \rangle$ and $J \subseteq R$,
there exist $w \in \langle J \rangle$
and $J_1 \subseteq J$ such that $J_1 \neq \emptyset$ and $Y^w = \langle J_1 \rangle$.
By Corollary \ref{finitesubgroupcor} the group $\langle Y,x \rangle$ is a finite group
and therefore $\langle Y^w,x^w \rangle$ is finite as well.

Let $r \in J_1 \subseteq \langle J_1 \rangle = Y^w$.
Then $\langle r,x^w \rangle \leq \langle Y^w,x^w \rangle$ is a finite group
and therefore $x^w$ normalizes $\langle J \rangle$ by Lemma \ref{ruhasinfiniteorder}.
As $w \in \langle J \rangle$ it follows that $x$ normalizes $\langle J \rangle$
and hence also $Z(\langle J \rangle) = \langle s \rangle$.
Applying Lemma \ref{rightangledcentralizer} we obtain $x \in \langle \{ s \} \cup s^{\perp} \rangle$.
\end{proof}

\medskip
\noindent
{\bf Proof of Proposition \ref{barDprop}:} Let $c \in C$ and $u \in S$.
Then $1 \neq c \in \langle J \rangle$. Thus, if $cu$ has finite order,
then $\langle c,u \rangle$ is a finite group. By Lemma \ref{cuhasfiniteorder}
it follows that $u \in \langle \{ s \} \cup s^{\perp} \rangle$
and hence $u \in \{ s \} \cup s^{\perp}$ because $u \in S$. We conclude
that $cu$ has infinite order for all $u \in s^{\infty}$.

\section{On the Cayley graph of $(W,S)$}

The goal of this section is to prove the following proposition.

\begin{prop} \label{keyproposition}
Let $(W,S)$ be a Coxeter system, $\{ \tau, a \} \subseteq S^W$ be such that
$a \tau$ has finite order and such that
$b:= a^{\tau} \neq a$. Let $\sigma \in \langle a, \tau \rangle \cap S^W$
be such that $\sigma \neq \tau$ and $[\tau,\sigma] = 1$.

Suppose that $(U_0,U_1,\ldots,U_k)$ is a sequence of subgroups of
$W$ such that the following hold:

\begin{itemize}
\item[(i)] $\langle U_{i-1},U_i \rangle$ is a finite group for all $1 \leq i \leq k$;
\item[(ii)] $\langle U_i, \tau \rangle$ and $\langle U_i, \sigma \rangle$ are both
infinite groups for all $0 \leq i \leq k$.
\end{itemize}

Then there exists an $x \in \{ a,b \}$ such that
$\langle U_i,x \rangle$ is an infinite group for all $0 \leq i \leq k$.
\end{prop}

\noindent
{\bf Remark:} Our proof of the proposition uses the Cayley graph
associated with $(W,S)$ which is in fact the unique thin building
of type $(W,S)$. We shall apply several basic facts about
buildings in our reasoning. It is convenient to adapt
the language of buildings for the Cayley graph in order to
be able to give references in the literature. Thus, the vertices
of the Cayley graph will be called {\sl chambers}, the edges of
the Cayley graph will be called {\sl panels} and right-cosets
of standard parabolic subgroups will be called {\sl residues}.

\subsection*{Galleries and convexity in $\Sigma(W,S)$}

Let $(W,S)$ be a Coxeter system. The Coxeter complex associated
with $(W,S)$ is defined to be the pair
$\Sigma(W,S) :=({\cal C},{\cal P})$ where ${\cal C} := W$
and ${\cal P}:= \{ \{ sw,w \} \mid s \in S, w \in W \}$. The elements
of ${\cal C}$ are called the {\sl chambers} of $\Sigma(W,S)$
and the elements of ${\cal P}$ are called the {\sl panels} of $\Sigma(W,S)$.

As $S$ generates $W$ the graph $\Sigma(W,S)$ is connected. Let $c,d \in {\cal C}$.
A {\sl gallery from $c$ to $d$ of length $m \in {\bf N}$} is a sequence
$\gamma = (c = c_0,c_1,\ldots,c_m = d)$
such that $\{ c_{i-1},c_i \} \in {\cal P}$ for all $1 \leq i \leq m$.
The {\sl distance between $c$ and $d$} is the length of a gallery joining them
of minimal length; it is denoted by $\ell(c,d)$ and we observe that
$\ell(c,d) := \ell(cd^{-1})$ where $\ell:W \rightarrow {\bf N}$ denotes
the length function of $(W,S)$. Note that $({\cal C}, \ell)$
is a metric space and that therefore we have a natural notion of a {\sl convex subset
of ${\cal C}$}.

\subsection*{Residues in $\Sigma(W,S)$}

We continue to assume that $(W,S)$ is a Coxeter system and
we let $\Sigma(W,S) = ({\cal C},{\cal P})$ be its Coxeter complex.

A subset $R$ of ${\cal C}$ is called a {\sl residue} of $\Sigma(W,S)$
if there are a subset $J$ of $S$ and $w \in W$ such that $R = \langle J \rangle w$.
As $\langle J \cap K \rangle = \langle J \rangle \cap \langle K \rangle$
for all $J,K \subseteq S$, the set $J \subseteq W$ in the definition of
a residue $R$ is uniquely determined by $R$. This set is called the
{\sl type of $R$} and the {\sl rank of $R$} is defined to be the cardinality
of its type. We observe that the residues of rank 1 are precisely the panels
of $\Sigma(W,S)$. A residue is called {\sl spherical} if its type is spherical;
hence a residue $R$ is spherical if and only if $R$ is a finite set.

\begin{lemma} \label{residuesareconvex}
Let $R \subseteq {\cal C}$ be a residue. Then $R$ is a convex subset
of $({\cal C},\ell)$.
\end{lemma}

\begin{proof} This is Proposition 3.24 in \cite{lilaWeiss}.
\end{proof}

For a chamber $c \in {\cal C}$ and $w \in W$ we put $c^w := cw$ where
$cw$ is the product in $W$. In this way we get an action
${\cal C} \times W \rightarrow {\cal C}, (c,w) \mapsto c^w$
which is regular on the set of chambers and type-preserving on
the set of residues.

\smallskip
\noindent
{\bf Remark:} Note that in our setup, the
group $W$ acts from the right on $\Sigma(W,S)$.

\begin{lemma} \label{residuesbasic}
Let $R \subseteq {\cal C}$ be a residue, let $J \subseteq S$
be its type and let $w \in R$. Let ${\cal P}_R := \{ P \in {\cal P} \mid P \subseteq R \}$.
 Then the following hold:
\begin{itemize}
\item[(i)] the stabilizer of $R$ in $W$ is the group $w^{-1} \langle J \rangle w$;
\item[(ii)] the map $x \mapsto x^w$ is an isomorphism from $\Sigma(\langle J \rangle,J)$
onto $\Sigma_R := (R,{\cal P}_R)$;
\item[(iii)] if $|J| = 2$ and ${\cal E} \subseteq {\cal P}_R$ has cardinality at least 3,
then the graph $(R, {\cal P}_R \setminus {\cal E})$ has at least 3 connected components.
\end{itemize}
\end{lemma}

\begin{proof} Assertions (i) and (ii) are straightforward.
Suppose that $|J| = 2$ which means that $\langle J \rangle$ is a dihedral group.
Then $\Sigma(\langle J \rangle, J)$
is isomorphic to $({\bf Z},\{ \{ z,z+1 \} \mid z \in {\bf Z} \})$
(if $|\langle J \rangle| = \infty$) or to a circuit
of length $2k$ for some $2 \leq k \in {\bf N}$ (if $\langle J \rangle$
is a finite group). In both cases one verifies that removing at least
three edges produces at least three connected components and
Assertion (iii) is thus a consequence of Assertion (ii).
\end{proof}

\subsection*{Walls and roots in $\Sigma(W,S)$}

We continue to assume that $(W,S)$ is a Coxeter system and
we let $\Sigma(W,S) = ({\cal C},{\cal P})$ be its Coxeter complex.
We recall that $S^W := \{ w^{-1}sw \mid w \in W, s \in S \}$
is the set of reflections of $(W,S)$.

\begin{lemma}
Let $P \in {\cal P}$. Then $Stab_W(P) = \langle t \rangle$ for
some $t \in S^W$.
\end{lemma}

\begin{proof} This follows from Assertion (i) of Lemma \ref{residuesbasic}
and the fact, that panels are precisely the residues of rank 1.
\end{proof}

Let $t \in S^W$ be a reflection of $(W,S)$. The {\sl wall of $t$}
is defined to be the set $M_t$ of all panels stabilized by $t$;
hence $M_t := \{ P \in {\cal P} \mid P^t = P \}$. Furthermore,
we define the graph $\Sigma_t := ({\cal C}, {\cal P} \setminus M_t)$.

\begin{lemma} \label{rootdeflemma}
Let $t \in S^W$ be a reflection. Then $\Sigma_t$ has two connected
components which are interchanged by $t$. If $u \in S^W$
is a reflection distinct from $t$, then $M_t \cap M_u = \emptyset$;
in particular, each panel $Q \in M_u$ is contained in one of the
two connected components of $\Sigma_t$.
\end{lemma}

\begin{proof} This follows from Propositions 3.11 and 3.12 in \cite{lilaWeiss}.
\end{proof}

Let $t \in S^W$ be a reflection of $(W,S)$. The two connected components
of $\Sigma_t(W,S)$ are called the {\sl roots associated with $t$}. For a chamber
$c \in {\cal C}$ we denote the root associated to $t$ which contains $c$ by
$H(t,c)$; more generally, if $\emptyset \neq X \subseteq {\cal C}$
is such that $H(t,x) = H(t,y)$ for all $x,y \in X$, then we denote the unique
root associated with $t$ containing the set $X$ by $H(t,X)$.

A {\sl root} of $(W,S)$ is a set of chambers $\emptyset \neq \alpha \subseteq {\cal C}$
such that there exists a reflection $t \in S^W$ with $\alpha = H(t,\alpha)$
and $\Phi(W,S)$ denotes the set of all roots of $(W,S)$.

\begin{prop} \label{rooffundametalprop}
The following hold:
\begin{itemize}
\item[(i)] Roots are convex subsets of $({\cal C},\ell)$.
\item[(ii)] if $X \subseteq {\cal C}$ is a convex subset of
$({\cal C},\ell)$ then $X$ is the intersection of all roots containing $X$.
\item[(iii)] Suppose that $X \subseteq {\cal C}$ is a connected subset
of $\Sigma(W,S)$ and that $t \in S^W$ is such that $X$ contains no panel
on the wall of $t$. Then there exists a root associated with $t$ containing $X$.
In particular, $H(t,X)$ is well defined.
\item[(iv)] Let $R \subseteq {\cal C}$ be a residue and let $t \in S^W$
be a reflection. Then $t$ stabilizes $R$ if and only if there is a panel
$P \in M_t$ such that $P \subseteq R$. If this is not the case, then $R$
is contained in a unique root associated with $t$, i.e. $H(t,R)$ is well
defined.
\end{itemize}
\end{prop}

\begin{proof} Assertion (i) is Proposition 3.19 in \cite{lilaWeiss}
Assertion (ii) is (29.20) in \cite{orangeWeiss}.

Let $(x_0,x_1,x_2)$ be path of length 2 in $X$. Then the panels
$P := \{x_0,x_1 \}$ and $Q := \{ x_1,x_2 \}$ are not in $M_t$ by our
assumption and hence $H(t,P)$ and $H(t,Q)$ is well defined.
It follows that $H(t,x_0) = H(t,P) = H(t,x_1) = H(t,Q) = H(t,x_2)$.
By induction of the length of a path in $X$ joining two chambers
$x$ and $y$ in $X$ it follows that $H(t,x) = H(t,y)$ for any two chambers
in $X$. Thus Assertion (iii) holds.

Let $R \subseteq {\cal C}$, let $J \subseteq R$ be its type  and $t \in S^W$. If there exists
a panel $P \in M_t$ which is contained in $R$, then $P$ is stabilized
by $t$ and since $R$ is the unique $J$-residue containing $P$, $R$ is stabilized
by $t$ as well. Suppose now that $R$ does not contain a panel of wall of $t$.
Since $R$ is convex, it is connected and Assertion (iii) yields that
$H(t,R)$ is well defined. Now $H(t,R^t) = (H(t,R))^t = {\cal C} \setminus H(t,R)$
and hence $R^t \neq R$.
\end{proof}

\begin{lemma} \label{rootstabilinglemma}
Let $t \neq u \in S^W$ be such that $tu = ut$.
Then $t$ stabilizes both roots associated with $u$.
\end{lemma}

\begin{proof} Let $P$ be a panel in $M_t$.
As $t \neq u$ and $P \in M_t$ the root $H(u,P)$ is well defined
and we have $(H(u,P))^t = H(u^t,P^t) = H(u,P)$. Hence $t$ stabilizes $H(u,P)$
and hence also $-H(u,P) := {\cal C} \setminus H(u,P)$.
\end{proof}

\begin{lemma} \label{atmosttwopanels}
Let $t \in S^W$ and let $R \subseteq {\cal C}$ be a residue of rank 2.
Then $|\{ P \in M_t \mid P \subseteq R \}| \leq 2$.
\end{lemma}

\begin{proof} Let $\alpha$ and $-\alpha$ be the two roots associated with
$t$. As $R$ is convex (by Lemma \ref{residuesareconvex}) and
as roots are convex (by Assertion (i) of Proposition \ref{rooffundametalprop}),
it follows that $R \cap \alpha$ and $R \cap -\alpha$ are convex and in particular
connected. Setting, as in Lemma \ref{residuesbasic},
${\cal P}_R := \{ P \in {\cal P} \mid P \subseteq R \}$ it follows that
the graph $(R, {\cal P}_R \setminus M_t)$ has at most two connected
components. Thus Assertion (iii) of Lemma \ref{residuesbasic} yields
$|\{ P \in M_t \mid P \subseteq R \}| \leq 2$ and we are done.
\end{proof}

\subsection*{Projections in $\Sigma(W,S)$}

We continue to assume that $(W,S)$ is a Coxeter system and
we let $\Sigma(W,S) = ({\cal C},{\cal P})$ be its Coxeter complex.

\begin{lemma} \label{projectionlemma}
Let $R \subseteq {\cal C}$ be a residue and $c \in {\cal C}$.
Then there exists a unique chamber $d \in R$ such that
$\ell(c,x) = \ell(c,d) + \ell(d,x)$ for all $x \in R$.
\end{lemma}

\begin{proof} This is Theorem 3.22 in \cite{lilaWeiss}.
\end{proof}

Let $R \subseteq {\cal C}$ be a residue and $c \in {\cal C}$.
The unique chamber $d$ in Lemma \ref{projectionlemma} is
called {\sl the projection of $c$ onto $R$} and it will be
denoted by $proj_R c$. For any subset $X$ of ${\cal C}$
we put $proj_R X := \{ proj_R x \mid x \in X \}$.

\begin{lemma} \label{projectionofpanelsonresidues}
Let $R$ be a residue and $t \in S^W$ be such that
$R^t = R$. Then the following hold:
\begin{itemize}
\item[(i)] $H(t,c) = H(t,proj_R c)$ for each chamber $c \in {\cal C}$;
\item[(ii)] $proj_R P \in M_t$ for all $P \in M_t$.
\end{itemize}
\end{lemma}

\begin{proof} $H(t,c)$ is a convex set of chambers by
Assertion (i) of Proposition \ref{rooffundametalprop}. As $R^t = R$, there is a panel $P \in M_t$ which is contained
in $R$ by Assertion (iv) of Proposition \ref{rooffundametalprop} and hence there is a
chamber $d \in R \cap H(t,c)$. As there is a minimal gallery from $c$ to $d$
passing through $proj_R c$ we have $proj_R c \in H(t,c)$ which yields
$H(t,c) = H(t,proj_R c)$ and hence Assertion (i).

Let $x,y$ be the two chambers in $P$ and let $x' := proj_R x$, $y' := proj_R y$.
Without loss of generality we may assume that $\ell(y,y') \leq \ell(x,x')$.
Assume, by contradiction, that $\ell(x',y') \geq 2$. Then $\ell(x,y') =
\ell(x,x') + \ell(x',y') \geq \ell(x,x') + 2 \geq \ell(y,y') + 2 = \ell(y,y') + \ell(x,y) +1
> \ell(x,y')$. Thus $\ell(x',y') \leq 1$. As $R^t = R$ and $x^t = y$,
we have $x'^t = y' \neq x'$ and hence also $y'^t = x'$.
Thus $Q:= \{ x',y' \}$ is a panel contained in $R$ and stabilized by $t$.
As $Q = proj_R P$ this finishes the  proof of Assertion (ii).
\end{proof}

\begin{prop} \label{wallcontainedintersection}
Let $R \subseteq {\cal C}$ be a residue of rank 2 and let
$t,u,v \in S^W$ be pairwise distinct reflections such that
$R^t = R^u = R^v =R$ and $uv = vu$. Then there exist roots
$\alpha,\beta \in \Phi(W,S)$ such that the following holds:

\begin{itemize}
\item[(i)] $\alpha$ is associated with $u$ and $\beta$ is associated with $v$;
\item[(ii)] any panel $P \in M_t$ is contained in $(\alpha \cap \beta) \cup (-\alpha \cap -\beta)$.
\end{itemize}

If $\alpha$ and $\beta$ are as above, then
any panel in $M_{utu}$ is contained in $(\alpha \cap -\beta) \cup (-\alpha \cap \beta)$.
\end{prop}

\begin{proof}
As $t$ stabilizes $R$ there exists a panel
$P \in M_t$ which is contained in $R$. Since $u \neq t \neq v$ the roots
$\alpha := H(u,P)$ and $\beta := H(v,P)$ are well defined. Since $u \neq v$
and $r:= uv = vu$ it follows that $r$ is in the center of the stabilizer
of $R$ in $W$ by Lemma \ref{2reflectionlemma}.
We have in particular $rt = tr$ and hence $r$ stabilizes the wall $M_t$ of $t$.
As $r$ stabilizes also the residue $R$, we have that $Q := P^r$ is a panel
in the wall of $t$ which is also contained in $R$.

Note also that $\alpha^r = (\alpha^v)^u = \alpha^u = -\alpha$ by Lemma \ref{rootstabilinglemma},
and similarly $\beta^r = -\beta$. As $Q = P^r  \subseteq (\alpha \cap \beta)^r = (-\alpha) \cap (-\beta)$
it follows in particular $Q \neq P$. It follows by Lemma \ref{atmosttwopanels}
that $\{ X \in M_t \mid X \subseteq R \} = \{ P,Q \}$. Let $Y \in M_t$ be a panel
on the wall of $t$. Then, by Assertion (ii) of Lemma \ref{projectionofpanelsonresidues}
we have $proj_R Y \in \{ P,Q \}$. If $proj_R Y = P$,
then $H(u,Y) = H(u,P) = \alpha$ and $H(v,Y) = H(v,P) = \beta$ (by Assertion (i)
of Lemma \ref{projectionofpanelsonresidues}) and hence $Y \subseteq \alpha \cap \beta$.
Similarly, we obtain $Y \subseteq (-\alpha) \cap (-\beta)$ if $proj_R Y = Q$.
This finishes the proof of the first assertion.

Let $P' := P^u$. As $P$ is a panel in the wall of $t$ we have $P' \in M_{utu}$
and as $u$ stabilizes $R$ and $P \subseteq R$, the panel $P'$ is also contained
in $R$. Now $H(u,P') = H(u,P^u) = (H(u,P))^u = \alpha^u = -\alpha$
and $H(v,P') = H(v,P^u) = (H(v^u,P))^u = (H(v,P))^u = \beta$.
It follows now from the first assertion that each panel in $M_{utu}$
is contained $((-\alpha) \cap \beta) \cup (\alpha \cap (-\beta))$ and we are done.
\end{proof}

\subsection*{Finite subgroups of $W$}

We continue to assume that $(W,S)$ is a Coxeter system and
we let $\Sigma(W,S) = ({\cal C},{\cal P})$ be its Coxeter complex.

\begin{lemma} \label{sphbasic}
Let $U \leq W$ be a finite subgroup of $W$. Then $U$ stabilizes
a spherical residue of $\Sigma(W,S)$.
\end{lemma}

\begin{proof} This follows from Lemma \ref{Bourbakiexercise} and Assertion (i) of Lemma \ref{residuesbasic}.
\end{proof}

For a finite subgroup $U \leq W$ we let $Sph(U)$ denote the set of all
spherical residues stabilized by $U$.

\begin{lemma} \label{sphkeylemma}
Let $U \leq W$ be a finite subgroup and let $t \in S^W$
be such that $\langle U,t \rangle$ is an infinite group.
Then there exists a unique root associated with $t$
which contains each residue in $Sph(U)$.
\end{lemma}

\begin{proof}
This is Lemma 2.6. in \cite{MW2002}.
\end{proof}

Let $U \leq W$ be a finite subgroup of $W$ and $t \in S^W$
be such that $\langle U,t \rangle$ is an infinite group.
Then the unique root associated with $t$ which contains each
spherical residue stabilized by $U$ is denoted by $H(t,U)$.

\begin{prop} \label{sequenceoffinitesubroups}
Let $t \in S^W$ and $(U_0,U_1,\ldots,U_k)$
be a sequence of subgroups such that
\begin{itemize}
\item $\langle U_{i-1},U_i \rangle$ is a finite group for all $1 \leq i \leq k$;
\item $\langle U_i,t \rangle$ is an infinite group for all $0 \leq i \leq k$.
\end{itemize}

Then $H(t,U_i) = H(t,U_j)$ for all $0 \leq i,j \leq k$.
\end{prop}

\begin{proof}
We proceed by induction on $k$.
For $k = 0$ the assertion is trivial and we may assume $k > 0$.
By induction it suffices to show that $H(t,U_{k-1}) = H(t,U_k)$.
As $V:= \langle U_{k-1}, U_k \rangle$ is a finite
subgroup of $W$ by assumption, there exists a spherical residue
$R$ stabilized by $V$ (by Lemma \ref{sphbasic}).
As $U_k \leq V$ and $\langle U_k,t \rangle$
is infinite, it follows that $\langle V,t \rangle$ is infinite
and $R \subseteq H(t,V)$; we have in particular, $H(t,V) = H(t,R)$.
As $U_{k-1}$ and $U_k$ are both subgroups of $V$, it follows that
$R$ is also stabilized by these groups. Hence we have
$H(t,U_{k-1}) = H(t,R) = H(t,U_k)$.
\end{proof}

\begin{cor} \label{keyargument}
Let $R \subseteq {\cal C}$ be a residue of rank 2 and let
$t,u,v \in S^W$ be pairwise distinct reflections such that
$R^t = R^u = R^v =R$ and $uv = vu$.
Let $(U_0,U_1,\ldots,U_k)$
be a sequence of subgroups such that
\begin{itemize}
\item $\langle U_{i-1},U_i \rangle$ is a finite group for all $1 \leq i \leq k$;
\item $\langle U_i,u \rangle$ and $\langle U_i,v \rangle$ are both
infinite groups for all $0 \leq i \leq k$.
\end{itemize}

Then there exists $x \in \{ t,utu \}$ such that
$\langle U_i,x \rangle$ is an infinite group for all $0 \leq i \leq k$.
\end{cor}

\begin{proof}
By Proposition \ref{sequenceoffinitesubroups} there exists a root $\gamma$ associated
to $u$ such that $\gamma = H(u,U_i)$ for all $0 \leq i \leq k$
and a root $\delta$ associated to $v$ such that $\delta = H(v,U_i)$
for all $0 \leq i \leq k$. Using Proposition \ref{wallcontainedintersection} we see that
there exists a reflection $x \in \{ t, utu \}$ such that $\gamma \cap \delta$
contains no panel of the wall $M_x$ associated with $x$. We claim
that $\langle U_i,x \rangle$ is an infinite group for all $0 \leq i \leq k$.

Assume, by contradiction, that $V := \langle U_i,x \rangle$ is a finite
subgroup of $W$. Then $V$ stabilizes a spherical residue $T$ by Lemma \ref{sphkeylemma}
and since $x \in V$, there exists a panel $P \in M_x$ which is contained
in $T$. On the other hand, $T$ is a spherical residue stabilized
by $U_i$ because $U_i \leq V$. Hence we have $T \subseteq H(u,U_i) = \gamma$
and $T \subseteq H(v,U_i) = \delta$ which yields $P \subseteq \gamma \cap \delta$.
Hence there exists $P \in M_x$ such that $P \subseteq \gamma \cap \delta$
which contradicts our choice of $x$.
\end{proof}

\bigskip
\noindent
{\bf Proof of Proposition \ref{keyproposition}:}
Since $a,\tau \in S^W$ and $a \tau$ has finite order,
there exists a rank 2 residue $R \subseteq {\cal C}$ in $\Sigma(W,S)$
which is stabilized by $\langle a,\tau \rangle$.
By the hypothesis of the proposition we have $b := a^{\tau} \neq a$,
$\sigma \in \langle a,\tau \rangle \cap S^W$, $\sigma \neq \tau$ and $[\sigma,\tau] =1$.
As $\sigma \in \langle a,\tau \rangle \leq Stab_W(R)$ we have
$R^a = R^{\tau} = R^{\sigma} = R$.
Also, since $a^{\tau} \neq a$ and $\sigma^{\tau} = \sigma$ we
have that $a, \tau$ and $\sigma$ are pairwise distinct.
Thus, applying Corollary \ref{keyargument} with $a := t, u := \tau$ and
$v := \sigma$ provides the proposition.

\section{The case $D$}

\label{sectioncaseD}

\noindent
{\bf Convention:} Throughout this section $(W,S)$ is a Coxeter system and $s \in S$
is a right-angled reflection of $(W,S)$.
Moreover,
$C \subseteq s^{\perp}$ is a $s$-component of type $D_{2k+1}$,
$\rho$ is the longest element of $(\langle C \rangle,C)$ and
$a \in C$ is such that $b:= \rho a \rho \neq a$. We put $\tau := s\rho$
and $\sigma := abs\rho = ab\tau$.
Finally, $R \subseteq W$
is a Coxeter generating set of $W$ and $J \subseteq R$ is of type $C_{2k+1}$ and such that
$\langle J \rangle = \langle s \rangle \times \langle C \rangle$ and
$J \subseteq (s\rho)^{\langle J \rangle} \cup a^{\langle J \rangle}$ and $a,\tau \in R^W$.

\medskip
\noindent
The goal of this section is to prove the following.

\begin{prop} \label{mainDodd}
The generator $a$ is a blowing down generator for $s$.
\end{prop}

\begin{lemma} \label{ConditionsforRank2CaseDPart}
The following hold.
\begin{itemize}
\item[(i)] $ab = ba \neq 1$ and $ab$ is an involution in $\langle s^{\perp} \rangle$;
\item[(ii)] $\rho ab = ab\rho \neq 1$ and $ab\rho$ is an involution in $\langle s^{\perp} \rangle$;
\item[(iii)] $\rho s = s\rho$, $\tau$ is an involution, $a^{\tau} = b$;
\item[(iv)] $\sigma = \tau^a$, $\sigma \neq \tau$ and $\sigma \tau = \tau \sigma$;
\item[(v)] if $u \in s^{\infty}$, then $\sigma u$ and $\tau u$ are both of infinite order;
\item[(vi)] $\sigma \in R^W$.
\end{itemize}
\end{lemma}

\begin{proof}
As $\rho$ is the longest element of $(\langle C \rangle,C)$ which is of type
$D_{2k+1}$ and $a \in C$ is
such that $b = a^{\rho} \neq a$, it follows that $ab = ba$ and as $a \neq b$ are
both involutions, it follows that $ab$ is an involution.
As $C$ is a spherical irreducible component of $(\langle s^{\perp} \rangle,s^{\perp})$
and $a,b \in C$, it follows that $ab \in \langle s^{\perp} \rangle$. This
concludes the proof of Assertion (i).

As $\rho$ is an involution and $a^{\rho} =b$, we have $b^{\rho} = a$ and
hence $(ab)^{\rho}= ba = ab$ where the last equality follows from Assertion (i).
Thus $ab$ and $\rho$ are commuting involutions which shows that $(\rho ab)^2 = 1$;
As $(\langle C \rangle,C)$ is of type
$D_{2k+1}$ with $1 \leq k$, it follows that the length of $\rho$ with respect
to the generator set $C$ is $2k(2k+1) \geq 6$. On the other hand we have
the length of $ab$ with respect to $C$ is 2 because $a \neq b \in C$ and therefore
$\rho \neq ab$ which shows that $ab \rho$ is an involution. As $\rho,a,b \in \langle C \rangle$
and $C$ is an irreducible spherical component of $(\langle s^{\perp} \rangle,s^{\perp})$,
it follows that $ab \rho \in \langle s^{\perp} \rangle$. This concludes
the proof of Assertion (ii).

As $\rho \in \langle s^{\perp} \rangle$ and $s \not \in \langle s^{\perp} \rangle$
it follows that $s \neq \rho$ and $s\rho = \rho s$. As $s$ and $\rho$ are both involutions,
$\tau = s\rho $ is an involution as well. As $a \in s^{\perp}$ we have
$[a,s] = 1$. It follows $a^{\tau} = \tau a \tau = (s\rho)a(s\rho) = a^{\rho} = b$
which finishes (iii).

We have $\tau^a = (s\rho)^a = as\rho a = sa \rho a = s\rho (\rho a\rho )a = s\rho a^{\rho}a = s\rho ba = \sigma$,
since we have already established  $[a,s] = [\rho,s] = [a,b] = 1$ in the previous
parts of this proof. We have also $\tau^s = s(s\rho )s = \rho s = s\rho = \tau$.
Assume by contradiction that $\tau = \sigma$. Then
$s\rho ab = s\rho$ and hence $ab =1$ implying $a = b$ which yields a contradiction.
Finally, since $[a,s] = [b,s] = [ab,\rho] =1$
we have $\sigma \tau = s\rho abs\rho = s\rho s(ab)\rho = s\rho s\rho ab = \tau \sigma$
and we are done with (iv).

We first remark that
$\rho$ and $\rho ab$ are both involutions in $\langle s^{\perp} \rangle$
by Assertion (ii). As $\tau = s\rho$ and $\sigma = s\rho ab$, Assertion (v) follows
from Corollary \ref{infordercor}.

Finally, Assertion (vi) follows from Assertion (iv) of the Lemma.
\end{proof}

\begin{prop} \label{blowingdownDprop}
Let $(u_0,u_1,\ldots,u_n)$ be a sequence in $s^{\infty}$ such that
$u_{i-1}u_i$ has finite order for all $1 \leq i \leq n$.
There exists
$x \in \{ a,b \}$ such that $xu_i$ has infinite order for
all $0 \leq i \leq n$.
\end{prop}

\begin{proof} We have $\{ \tau,a \} \subseteq R^W$ .
Moreover $a\tau$ has finite order since $a$ and $\tau$ are both contained
in the finite subgroup $\langle J \rangle$. We have also
$b = a^{\tau} \neq a$ and Assertion (iv)
of Lemma \ref{ConditionsforRank2CaseDPart} yields $\sigma = \tau^a$
and hence $\sigma \in \langle \tau,a \rangle \cap R^W$ by Assertion (vi) of
Lemma \ref{ConditionsforRank2CaseDPart}. By Assertion (iv) of
Lemma \ref{ConditionsforRank2CaseDPart} we have
$\sigma \neq \tau$ and $[\tau,\sigma] =1$.

Let $U_i := \langle u_i \rangle$ for $0 \leq i \leq n$. Then $\langle U_{i-1},U_i \rangle
= \langle u_{i-1},u_i \rangle$
is a finite group for $1 \leq i \leq n$ because the $u_i$ are involutions
and $u_{i-1}u_i$ has finite order by hypothesis. Finally, Assertion (v) of
Lemma \ref{ConditionsforRank2CaseDPart} yields that $\langle \tau,U_i \rangle$
and $\langle \sigma,U_i \rangle$ are infinite groups for $0 \leq i \leq n$
because $u_i \in s^{\infty}$.

We are now in the position to apply Proposition \ref{keyproposition} with $S:= R$.
It asserts that there is an element $x \in \{ a,b \}$
such that
$\langle x,U_i \rangle$ is an infinite
group for all $0 \leq i \leq n$. As $U_i = \langle u_i \rangle$ with an involution $u_i$
for $0 \leq i \leq n$ it follows that there exists an $x \in \{ a,b \}$
such that $xu_i$ has infinite order for $0 \leq i \leq n$.
\end{proof}

\bigskip
\noindent
{\bf Proof of Proposition \ref{mainDodd}:} The generator $a$
satisfies Axiom (BDG1) by the general assumptions of this section and
Proposition \ref{blowingdownDprop} yields that $a$ satisfies Axiom (BDG2)
as well.

\section{The case $I$}

\label{sectioncaseI}

\noindent
{\bf Convention:} Throughout this section $(W,S)$ is a Coxeter system, $s \in S$
is a right-angled reflection of $(W,S)$.
Moreover,
$C = \{ a,b \} \subseteq s^{\perp}$ is a $s$-component of type $I_2(2k+1)$,
$\rho$ is the longest element of $(\langle C \rangle,C)$ and $\tau := s\rho$.
Finally, $R \subseteq W$
is a Coxeter generating set of $W$ and $J \subseteq R$ is of type $I_2(4k+2)$ and such that
$\langle J \rangle = \langle s \rangle \times \langle C \rangle$ and
$J \subseteq (s\rho)^{\langle J \rangle} \cup a^{\langle J \rangle}$ and $a,\tau \in R^W$.

\medskip
\noindent
The goal of this section is to prove the following.

\begin{prop} \label{mainI2odd}
The generator $a$ is a blowing down generator for $s$.
\end{prop}

Similarly as in the case $D$, the proof of this proposition consists essentially
of checking the conditions of Proposition \ref{keyproposition}. However, in the
$I_2$-case there is an additional difficulty which requires some extra work.
It is the proof of Lemma \ref{FACapplication} where we need an additional argument
with respect to the $D$-case. This lemma could be established at much lesser cost
if we would exclude the case $k=1$.

\subsection*{Coxeter systems and $FA$-groups}

The following definition is due to Serre (see Paragraph 6.1 in \cite{tree}):

\smallskip
\noindent
{\bf Definition:} A group $G$ is called an {\sl $FA$-group} if it satisfies the following
condition:

\begin{itemize}
\item[(FA)] If $G$ acts without inversion on a non-empty tree $T = (V,E)$,
then $G$ fixes a vertex $v \in V$.
\end{itemize}

\begin{lemma} \label{finiteimpliesFA}
Finite groups are $FA$-groups.
\end{lemma}

\begin{proof}
This is a special case of Example 6.3.1 in  \cite{tree}.
\end{proof}

\smallskip
\noindent
{\bf Definition:} Let $(W,S)$ be a Coxeter system of finite rank. A subset $J$ of $S$
is called {\sl 2-spherical} if $st$ has finite order for all $s,t \in J$.
A parabolic subgroup $P$ of $(W,S)$ is called 2-spherical, if
$P = \langle J \rangle^w$ for some 2-spherical subset of $S$ and some $w \in W$.

\smallskip
\noindent
The following result is due to Mihalik and Tschantz.

\begin{prop} \label{MihalikTschantz}
Let $(W,S)$ be a Coxeter system of finite rank. Then the following hold:
\begin{itemize}
\item[(i)] If $J \subseteq S$ is 2-spherical, then $\langle J \rangle$ is a $FA$-group.
\item[(ii)] If $U \leq W$ is a finitely generated $FA$-group, then $Pc_S(U)$ is
2-spherical.
\end{itemize}
\end{prop}

\begin{proof} Assertion (i) is Proposition 24 and Assertion (ii) is Lemma 25 in
\cite{MT}.
\end{proof}

\begin{cor} \label{MTcorollary}
Let $W$ be a Coxeter group and let $S \subseteq W$ and
$R \subseteq W$ be Coxeter generating sets. Let $J \subset S$ be a finite 2-spherical subset of $S$ and
let $s \in S$ be such that $s \in Pc_R(\langle J \rangle)$. Then $J \cup \{ s \}$ is also
a 2-spherical subset of $S$.
\end{cor}

\begin{proof} As $J$ is finite and 2-spherical, $\langle J \rangle$ is a finitely generated $FA$-group.
By Assertion (ii) of Proposition \ref{MihalikTschantz}
$P:= Pc_R(\langle J \rangle)$ is a 2-spherical parabolic subgroup of $(W,R)$.
(Note that $P = Pc_{R'}(\langle J \rangle)$ for a finite subset $R'$ of $R$, since $J$ is finite.)
It is
of finite rank and in particular a finitely generated $FA$-group. Moreover, $s \in P$ by assumption.
Again by Assertion (ii) of Proposition \ref{MihalikTschantz} $Q := Pc_S(P)$ is a 2-spherical parabolic subgroup of $(W,S)$
containing $\langle \{ s \} \cup J \rangle$. Thus there exists a 2-spherical subset $K$ of
$S$ and an element $w \in W$ such that $(\{ s \} \cup J)^w \leq \langle K \rangle$.
Hence, by Assertion (ii) of Proposition \ref{FormerLemma7} $ \{ s \} \cup J$ is a 2-spherical subset
of $S$.
\end{proof}

\subsection*{Finite paths in $s^{\infty}$}

\begin{lemma} \label{ConditionsforRank2CaseIPart1}
We have $b = a^{\tau}$, $\langle J \rangle = \langle a,\tau \rangle$,
$\rho \neq \tau \neq a \neq \rho \in R^W \cap \langle J \rangle$
and $[\tau,\rho] = 1$.
\end{lemma}

\begin{proof} As $a,b \in s^{\perp}$ and $\tau = s\rho$ we have $a^{\tau} = a^{s\rho} = a^{\rho} = b$
where the last equality follows from the fact that $\rho$ is the longest element of the
system $(\langle a,b \rangle, \{ a,b \})$ which is of type $I_2(2k+1)$.

As, $a$ and $\rho$ are in $\langle C \rangle$ and $\tau =s \rho$ it follows that
$\langle a, \tau \rangle \leq \langle s \rangle \times \langle C \rangle = \langle J \rangle$.
On the other hand, $b = a^{\tau} \in \langle a,\tau \rangle$ which implies
that $\langle C \rangle \leq \langle a,\tau \rangle$. We have in particular,
$\rho \in \langle a,\tau \rangle$ and therefore also $s = \tau \rho \in \langle a,\tau \rangle$
which implies that $\langle J \rangle = \langle s \rangle \times \langle C \rangle \leq \langle a,\tau
\rangle$ which finishes the proof of the second equation.
As $b = a^{\tau}$ it follows that $b \in \langle J \rangle \cap R^W$
and in particular that $\rho \in \langle a,b \rangle \leq \langle J \rangle$.
As $(\langle a,b \rangle, \{ a,b \})$ is of type $I_2(2k+1)$ and $\rho$ is its longest element,
it follows that $\rho \in a^W$ and hence $\rho \in \langle J \rangle \cap R^W$.
As $a^{\tau} = b \neq a$ it follows that $a \neq \tau$ and as $\tau \rho = s$ and $\rho^2 = 1$
it follows that
$\tau \neq \rho$. Finally, as $\rho$ is the longest element of the system
$(\langle a,b \rangle, \{ a,b \})$,
it follows that $a \neq \rho$.
As $\tau \in R$ we have $\tau^2 = 1$ and as $a^{\tau} =b$ we have $b^{\tau} = a$. It follows that
$\tau$ normalizes $\{ a,b \}$ and hence centralizes the longest element of the system
$(\langle a,b \rangle, \{ a,b \})$ which is $\rho$.
\end{proof}

\begin{lemma} We have $Pc_R(\langle a,b \rangle) = \langle J \rangle$ and
in particular $s \in Pc_R(\langle a,b \rangle)$.
\end{lemma}

\begin{proof} As $\langle a,b \rangle = \langle C \rangle \leq \langle J \rangle$,
the group $Pc_R(\langle a,b \rangle)$ is a parabolic subgroup of $(W,R)$
contained in $\langle J \rangle$
and $\langle J \rangle$ is a parabolic subgroup of $(W,R)$ of rank 2.
As the order of parabolic subgroup of rank one is 2 and the order of
$\langle a,b \rangle$ is 4k+2, it follows that
$Pc_R(\langle a,b \rangle) = \langle J \rangle$.
\end{proof}

\begin{lemma} \label{FAfiniteorder}
Let $u \in S \setminus \{ a,b,s \}$ be such that $au$ and $bu$ have finite order.
Then $u \in s^{\perp}$.
\end{lemma}

\begin{proof} By the previous lemma we have $s \in Pc_R(\langle a,b \rangle)$
As $au$ and $bu$ are of finite order the set $K := \{ a,b,u \}$ is a 2-spherical subset of $S$.
Thus we can apply Corollary \ref{MTcorollary} to see that
$\{ s \} \cup K$ is 2-spherical. As $s$ is a right-angled generator it follows $[s,u] =1$.
\end{proof}

\begin{lemma} \label{FACapplication}
Let $u \in s^{\infty}$. Then at least one of the elements
$ua$ and $ub$ has infinite order.
\end{lemma}

\begin{proof}
This follows from Lemma \ref{FAfiniteorder}.
\end{proof}

\begin{lemma} \label{ConditionsforRank2CaseI2Part2}
For each $u \in s^{\infty}$ the orders of $u\tau$ and $u\rho$ are infinite.
\end{lemma}

\begin{proof}
We have $a,b \in s^{\perp}$ and therefore $\rho \in \langle a,b \rangle$ is
an involution contained in  $\langle s^{\perp} \rangle$.
By Corollary \ref{infordercor} it follows that $\tau u = (s\rho )u$ has infinite order for each $u \in s^{\infty}$.

We have $a,b,u \in S$ and by Lemma \ref{FACapplication} we know at least one of
the elements $au$ and $bu$ has infinite order. As $\rho$ is the longest element in the
system $(\langle a,b \rangle, \{ a,b \})$, one  verifies using
the geometric representation (or the solution of the word problem
in Coxeter groups) that the order of $\rho u$ is also infinite.
\end{proof}

\begin{prop} \label{blowingdownIprop}
Let $(u_0,u_1,\ldots,u_n)$ be a sequence in $s^{\infty}$ such that
$u_{i-1}u_i$ has finite order for all $1 \leq i \leq n$.
There exists
$x \in \{ a,b \}$ such that $xu_i$ has infinite order for
all $0 \leq i \leq n$.
\end{prop}

\begin{proof} We have $\{ \tau,a \} \subseteq R^W$.
Moreover the order of $a\tau$ is finite since $a$ and $\tau$ are both contained
in the finite subgroup $\langle J \rangle$. We have also
$b = a^{\tau} \neq a$. By Lemma \ref{ConditionsforRank2CaseIPart1}
we have also $\rho \in \langle J \rangle \cap R^W$ and $[\tau,\rho] = 1$.

Let $U_i := \langle u_i \rangle$ for $0 \leq i \leq n$. Then $\langle U_{i-1},U_i \rangle
= \langle u_{i-1},u_i \rangle$
is a finite group for $1 \leq i \leq n$ because the $u_i$ are involutions
and $u_{i-1}u_i$ has finite order by hypothesis.
Moreover, Lemma \ref{ConditionsforRank2CaseI2Part2}
yields that $\langle \tau,U_i \rangle$
and $\langle \rho,U_i \rangle$
are infinite groups for $0 \leq i \leq n$
because $u_i \in s^{\infty}$.

We are now in the position to apply Proposition \ref{keyproposition} with $S:= R$ and $\sigma:= \rho$.
It asserts that there is an element $x \in \{ a,b \}$
such that
$\langle x,U_i \rangle$ is an infinite
group for all $0 \leq i \leq n$. As $U_i = \langle u_i \rangle$ with an involution $u_i$
for $0 \leq i \leq n$ it follows that there exists an $x \in \{ a,b \}$
such that $xu_i$ has infinite order for $0 \leq i \leq n$.
\end{proof}

\bigskip
\noindent
{\bf Proof of Proposition \ref{mainI2odd}:} The generator $a$
satisfies Axiom (BDG1) by the general assumptions of this section and
Proposition \ref{blowingdownIprop} yields that $a$ satisfies Axiom (BDG2)
as well.

\section{Proof of the main result}

We first recall the Proposition of the introduction.

\begin{prop} \label{mainresultdiffultdirection}
Let $(W,S)$ be a Coxeter system of arbitrary rank and
let $s \in S$ be a right-angled generator such that each $s$-component
has trivial center. If there exists a Coxeter generating set $T$ of $W$
such that $s$ is not a reflection of $(W,T)$,
then there exists a blowing down generator for $s$.
\end{prop}

\begin{proof} In view of the hypothesis of the Proposition
we are in the position to apply Proposition \ref{reductiontothreecases}.
Thus, there is a Coxeter generating set $R$ of $W$, an irreducible subset
$J$ of $R$ of $(-1)$-type and a $s$-component $C$ such that we are
in one of the cases $I,D$ or $\bar{D}$ described in Proposition \ref{reductiontothreecases}.
If we are in case $I$, (resp. $D$,$\bar{D}$) Proposition \ref{mainI2odd} (resp.
\ref{mainDodd}, \ref{barDprop}) asserts that there exists a blowing down generator for $s$.
\end{proof}

\smallskip
\noindent
The first assertion of the main result follows from Propositions \ref{mainresulteasydirectionprop} and
\ref{mainresultdiffultdirection}.
The second assertion follows from Lemma \ref{rightangled2}
applied to the Coxeter system $(W,R)$.

%

\end{document}